\documentclass[a4paper,12pt,onecolumn]{article}
\usepackage{etex}

\usepackage[vmargin=2cm,hmargin=2cm,headheight=14.5pt,top=2cm,headsep=.5cm]{geometry}

\usepackage{bm}
\usepackage{empheq}
\usepackage{stackrel}
\usepackage{cases}
\usepackage{mathtools}
\usepackage{amsthm,amsmath,amscd}
\usepackage{makeidx}
\usepackage[all]{xy}
\usepackage{perpage}
\usepackage{tikz-cd}
\tikzcdset{every label/.append style = {font = \small}}
\usepackage[symbol]{footmisc}

\MakePerPage[2]{footnote}
\usepackage{graphicx}
\DeclareMathSizes{12}{12}{8}{6}

\usepackage{cite}
\usepackage{url}
\usepackage[charter]{mathdesign}
\usepackage{accents}
\usepackage{hyperref}

\newtheoremstyle{ptheorem}{1em}{0em}{\itshape}{}{\bfseries}{.}{.5em}{\thmname{#1}\thmnumber{ #2}\thmnote{ (\hspace{-.01pt}{#3})}}

\theoremstyle{ptheorem}

\newtheorem{thm}{Theorem}[section]
\newtheorem{pro}[thm]{Proposition}
\newtheorem{lem}[thm]{Lemma}
\newtheorem{cor}[thm]{Corollary}

\newtheoremstyle{hdef}{1em}{0em}{}{}{\bfseries}{.}{.5em}{\thmname{#1}\thmnumber{ #2}\thmnote{ (\hspace{-.01pt}{#3})}}
\theoremstyle{hdef}

\newtheorem{dfn}[thm]{Definition}
\newtheorem{rem}[thm]{Remark}

\makeatletter
\newtheoremstyle{premark}{1em}{0em}{
}{}{\scshape}{.}{.5em}{}
\makeatother

\theoremstyle{premark}

\newtheorem{exa}[thm]{Example}

\numberwithin{equation}{section}
\numberwithin{figure}{section}

\DeclareMathOperator{\im}{im}

\let\sp\relax
\DeclareMathOperator{\sp}{sp}
\DeclareMathOperator{\Id}{Id}

\DeclareMathOperator{\dif}{d}

\newcommand{\cC}{{\mathcal C}}

\newcommand{\cF}{{\mathcal F}}

\newcommand{\cL}{{\mathcal L}}
\newcommand{\cM}{{\mathcal M}}

\newcommand{\cX}{{\mathcal X}}

\newcommand{\bC}{{\mathbb C}}

\newcommand{\bN}{{\mathbb N}}

\newcommand{\bR}{{\mathbb R}}

\newcommand{\bZ}{{\mathbb Z}}

\renewcommand{\a}{\alpha}
\renewcommand{\b}{\beta}
\renewcommand{\c}{\gamma}
\renewcommand{\l}{\lambda}

\renewcommand{\phi}{\varphi}

\newcommand{\<}{\langle}
\renewcommand{\>}{\rangle}

\newcommand{\ol}{\overline}

\renewcommand{\d}{\delta}

\renewcommand{\<}{\left<}
\renewcommand{\>}{\right>}

\renewcommand{\(}{\left(}
\renewcommand{\)}{\right)}

\newcommand{\til}{\tilde}

\newcommand{\olb}[1]{%
  \vbox{\offinterlineskip\ialign{\hfil##\hfil\cr $\rotatebox[origin=c]{90}{$]$}$\cr\noalign{\kern-.45ex}{$#1$}\cr}}}
\def\arraystretch{1.2}

\allowdisplaybreaks
\begin{document}
\title{On Linear Differential Equations and Systems\\ with Reflection\footnote{Partially supported by  Ministerio de Econom\'ia y Competitividad (Spain) project MTM2013-43014-P and Xunta de Galicia (Spain), project EM2014/032.}}

\author{
Alberto Cabada and F. Adri\'an F. Tojo\footnote{Supported by  FPU scholarship, Ministerio de Educaci\'on, Cultura y Deporte (Spain).} \\
\normalsize
Departamento de An\'alise Ma\-te\-m\'a\-ti\-ca, Facultade de Matem\'aticas,\\ 
\normalsize Universidade de Santiago de Com\-pos\-te\-la, Spain.\\ 
\normalsize e-mail: \textit{alberto.cabada@usc.es, fernandoadrian.fernandez@usc.es}\\
}
\date{}

\maketitle

\begin{abstract} In this paper we develop a theory of linear differential systems analogous to the classical one for ODEs, including the obtaining of fundamental matrices, the development of a variation of parameters formula and the expression of the Green's functions.  We also derive interesting results in the case of differential equations with reflection and generalize the Hyperbolic Phasor Addition Formula to the case of matrices.
\end{abstract}


\section{Introduction}

In recent years there have been a number of works concerning the study of differential problems with involutions. In the particular case of the reflection, starting with \cite{Cab4}, the computation of Green's functions for differential equations with reflection expanded \cite{Toj3,CTMal,CabToj,CabToj2,CTL,Sars,Sars2}. This has motivated several applications concerning almost periodic solutions \cite{Pia3, Pia4}, the obtaining of eigenvalues and explicit solutions of different problems \cite{Krits,Krits2} or their qualitative properties \cite{Ashy,Cab5}.

On the other hand, what happens in the case of linear differential systems with reflection and constant coefficients has  drawn far less attention \cite{Ait}. The authors intend, in this article, to provide some insight on this question. To do so, in Section \ref{sec2} we retake the original problem (for differential equations with reflection), providing two interesting results. First, we give an improvement on the general Reduction Theorem --see for instance \cite[Theorem 5.1.1]{CTL}, here Theorem \ref{thmdec}, that reduces the order of the resulting ODE --see Theorem \ref{hatdef}.  Second, we provide an explicit basis of the space of solutions of linear differential equations with reflection and constant coefficients --see Theorem \ref{strthm}. This knowledge is fundamental if our intention is to construct a fundamental matrix of the associated homogeneous problem.\par
This first part of the paper suggests that a similar attempt should be done in the case of systems of differential equations with reflections. Our approach will run in parallel to the classical theory of linear ODEs: construction of a fundamental matrix, description of the method of variation of parameters and, finally, obtaining of the associated Green's function. Unfortunately, this process will not be devoid of  difficulties. That is why in Section \ref{sec3} we will summarize some results concerning matrix functions which will be useful latter on.
\par
We start the study of systems of linear equations with reflection in Section \ref{sec4}. The first thing we do is to \textit{define} what a fundamental matrix is going to be in this setting. This is not obvious, for there are some properties that are satisfied in the case of systems of ODEs which will not apply here. For instance, contrary to our experience, a fundamental matrix \textit{may} be singular at some point of the real line. Once this definition is properly established, it is  time to derive the most basic results of the theory: those concerning existence and uniqueness of solution. Existence is derived from the analogous result for systems of ODEs --Lemma \ref{lemuos}, while existence is obtained through the direct construction of a fundamental matrix. This result --Theorem \ref{thmexpfm}, arguably one of the main results of the paper, expresses this fundamental matrix as a series of functional matrices. It is only under some extra assumptions that a simpler expression involving hyperbolic trigonometric functions may be found. The rest of the Section consists of rewriting this fundamental matrix in other ways. In order to achieve this, we have to generalize the Hyperbolic Phasor Addition Formula \cite[Lemma 1]{Toj4} to the algebra of matrices.\par
Section \ref{sec5} concerns the method of variation of parameters. Again, the method differs from the one we have in the case of ODEs. First we show that the classical approach does not work in this setting and then, studying a complementary problem, we arrive to a general method --Theorem~\ref{vpft}.\par
Finally, in Section \ref{sec6} we use the method of variation of parameters to obtain the Green's function both in the initial condition and the two point boundary condition cases. This is a natural generalization of the previous settings when concerning differential equations with reflections. To illustrate this point we recover, as shown in Example \ref{le}, the known expression of the Green's function for a first order periodic equation with reflection.
\par
\section{Differential equations with reflection}\label{sec2}
Let us introduce some definitions and notations. To start with, consider the differential operator $D$, the pullback operator of the reflection $\phi(t)=-t$, denoted by $\phi^*(u)(t):=u(-t)$, and the identity operator, $\Id$ (we will also denote by $\Id$ the identity matrix). \par
Let $T\in(0,+\infty)$ and $I:=[-T,T]$. We now consider the ring $\bR[D]$ of polynomials with real coefficients on the variable $D$ and the algebra $\bR[D,\phi^* ]$ consisting of the operators of the form
\begin{equation}\label{Lop}L:=\phi^*P(D)+Q(D),\end{equation}
where $P(D)=\sum_{k=0}^nb_kD^k$, $Q(D)=\sum_{k=0}^na_kD^k\in\bR[D]$ ($D^0=\Id$), $n\in\bN$, $a_k,b_k\in\bR,\ k=1,\dots,n$ which act as follows:
\begin{displaymath}Lu(t)=\sum_{k=0}^na_ku^{(k)}(t)+\sum_{k=0}^nb_ku^{(k)}(-t),\ t\in I,\end{displaymath}
on any function $u\in W^{n,1}(I)$.
The operation in the algebra $\bR[D,\phi^*]$ is the usual composition of operators (most of the time we will omit the composition sign). We observe that $D^k\phi^*=(-1)^k\phi^*D^k$ for $k=0,1,\dots$, which makes it a \textit{noncommutative algebra}. Actually, we have that $P(D)\phi^*=\phi^*P(-D)$ for any $P\in\bR[D]$.\par

 The following property is crucial for the obtaining of a Green's function.
\begin{thm}[{\cite[Theorem 2.1]{CTMal}}]\label{thmdec}
Take $L$ as defined in \eqref{Lop} and define
\begin{equation}\label{Rop}{L_1}:=\phi^*P(D)-Q(-D)\in\bR[D,\phi^* ].\end{equation}  Then ${L_1}L=L{L_1}\in\bR[D]$.
\end{thm}
\begin{rem}\label{remcoefred}
As it is pointed out in \cite{CTMal}, if ${L_1}L=\sum_{k=0}^{2n} c_kD^k$, then
\begin{displaymath}c_k=\begin{dcases} 0, & k \text{ odd,} \\
2\sum_{l=0}^{\frac{k}{2}-1}\(-1\)^l\(b_lb_{k-l}-a_la_{k-l}\)+\(-1\)^\frac{k}{2}\(b_\frac{k}{2}^2-a_\frac{k}{2}^2\), & k \text{ even.}\end{dcases}\end{displaymath}
\end{rem}
If $L=\sum_{i=0}^n \(b_i\phi^*+a_i\)D^i$ with $a_n\ne0$ or $b_n\ne0$, we have that $c_{2n}=(-1)^n(b_n^2-a_n^2)$. Hence, if $a_n=\pm b_n$, then $c_{2n}=0$. This shows that composing two elements of $\bR[D,\phi^* ]$ we can get another element with derivatives of less order.\par

We can improve Theorem~\ref{thmdec} in the following way. Let \[R(D):=\operatorname{mcd}(P(D),Q(D),P(-D),Q(-D)),\  \til P=P/R\text{ and }\til Q=Q/R.\]
Observe that $R$ is the polynomial constructed from the common roots, according to multiplicity, of $P(D)$, $Q(D),$ $P(-D)$ and $Q(-D)$.  $R(D)=R(-D)$, for if $\l$ is an root of $P(D)$ so has to be of $P(-D)$, and so $-\l$ has to be a root of $P(D)$. An important consequence of this is that $R$ commutes with $\phi^*$. We now have all it is needed for an improved version of Theorem~\ref{thmdec}.
\begin{thm}\label{hatdef}
	Take $L$, $R$, $\til P$ and $\til Q$ as previously defined and define
	\begin{equation}\label{Rop2}\widehat{L}:=\phi^*\til P(D)-\til Q(-D)\in\bR[D,\phi^* ].\end{equation}  Then $\widehat{L}L=L\widehat{L}\in\bR[D]$.
\end{thm}
\begin{proof}
\begin{align*}\widehat{L}L & =[\phi^*\til P(D)-\til Q(-D)][\phi^* P(D)+ Q(D)]=[\phi^*\til P(D)-\til Q(-D)][\phi^* \til P(D)+\til Q(D)]R(D)\\ & =[\til P(-D)\til P(D)+\phi^*\til P(D)\til Q(D)-\phi^*\til Q(D)\til P(D)-\til Q(-D)\til Q(D)]R(D)\\ & =[\til P(-D)\til P(D)-\til Q(-D)\til Q(D)]R(D).\end{align*}
On the other hand,
\begin{align*}L\widehat{L} & =[\phi^* P(D)+ Q(D)][\phi^*\til P(D)-\til Q(-D)]=[\phi^* \til P(D)+\til Q(D)]R(D)[\phi^*\til P(D)-\til Q(-D)]\\ & =[\til P(-D)\til P(D)+\phi^*\til Q(-D)\til P(D)-\phi^*\til P(D)\til Q(-D)-\til Q(-D)\til Q(D)]R(D)\\ & =[\til P(-D)\til P(D)-\til Q(-D)\til Q(D)]R(D).\end{align*}
\end{proof}

As we said before, $\widehat{L}L$ is a usual differential operator with constant coefficients. Consider now the following problem.
\begin{equation}\label{lccbvp}Su(t):=\sum_{k=0}^na_ku^{(k)}(t)=h(t),\ t\in I,\ B_ku:=\sum_{j=0}^{n-1}\left[\a_{kj}u^{j)}(-T)+\b_{kj}u^{(j)}(T)\right]=0,\ k=1,\dots,n.
\end{equation}


Given an operator $\cL$ for some set of functions of one variable, we will define the operator $\cL_\vdash$ as $\cL_\vdash G(t,s):=\cL(G(\cdot,s))|_{t}$ for every $s$ and any suitable function $G$ of two variables. We can then state the following result, which is a new version of the one found in \cite{CTMal} using the operator $\widehat L$ defined here. The proof is analogous.
\begin{thm}\label{thmdei} Consider the problem
\begin{equation}\label{rbvp}Lu(t)=h(t),\ t\in I,\ B_iu=0,\ k=1,\dots,n,
\end{equation}
where $L$ is defined as in \eqref{Lop}, $h\in L^1(I)$ and
\begin{displaymath}B_ku:=\sum_{j=0}^{n-1}\left[\a_{kj}u^{(j)}(-T)+\b_{kj}u^{(j)}(T)\right],\ k=1,\dots,n.\end{displaymath}
Then, there exists $\widehat{L}\in \bR[D,\phi^* ]$ --as in \eqref{Rop}-- such that $S:=\widehat{L}L\in\bR[D]$ and the unique solution of problem~\eqref{rbvp} is given by $\int_a^b\widehat{L}_\vdash G(t,s)h(s)\dif s,$ where $G$ is the Green's function associated to the problem $Su=0$, $B_k\widehat{L}u=0$, $B_ku=0$, $k=1,\dots,n$, assuming it has a unique solution.
\end{thm}
\subsection{The structure of solutions}
Let us consider now the operator $\widehat{L}L\in\bR[D]$. Observe that, by Remark~\ref{remcoefred}, $\widehat{L}L$ has only nonzero coefficients for even exponents of $D$. This means that $\widehat{L}L$ is of even degree, say $2m$ for $m\in\bN$, and that, if $z$ is root of $\widehat{L}L$, so it is $-z$. Hence, if $\widehat{L}L$ has $2 r$ real roots $\pm\l_1,\dots,\pm\l_r$ of multiplicities $\mu_1,\dots,\mu_r$ \footnote{Here one of the roots can be $0$, in which case has even multiplicity greater equal than 2.} and $2c$ complex roots $z_1=\pm(x_1+iy_1),\dots,z_c=\pm(x_c+iy_c)$ of multiplicities $\nu_1,\dots,\nu_2$, the $2m$-dimensional real vector space $V$ of solutions of $\widehat{L}Lu=0$ is generated by the basis of solutions\footnote{This basis would have to be rewritten in the case $0$ is a root in order to not repeat vectors.}
\begin{align*}& \hspace*{-.5em}\left\{  e^{\l_1 t},\dots, t^{\mu_1-1}e^{\l_1 t},e^{-\l_1 t},\dots, t^{\mu_1-1}e^{-\l_1 t},\dots,e^{\l_r t},\dots, t^{\mu_r-1}e^{\l_r t},e^{-\l_r t},\dots, t^{\mu_r-1}e^{-\l_r t}, e^{x_1t}\sin{y_1 t},\right.\\ & \left.e^{x_1t}\cos{y_1 t},\dots,t^{\nu_1-1}e^{x_1t}\sin{y_1 t},t^{\nu_1-1}e^{x_1t}\cos{y_1 t},e^{-x_1t}\sin{y_1 t},e^{-x_1t}\cos{y_1 t},\dots,t^{\nu_1-1}e^{-x_1t}\sin{y_1 t},\right.\\ & \left.t^{\nu_1-1}e^{-x_1t}\cos{y_1 t}, e^{x_ct}\sin{y_c t},e^{x_ct}\cos{y_c t},\dots,t^{\nu_c-c}e^{x_ct}\sin{y_c t},t^{\nu_c-c}e^{x_ct}\cos{y_c t},e^{-x_ct}\sin{y_c t},\right.\\ & \left.e^{-x_ct}\cos{y_c t},\dots,t^{\nu_c-c}e^{-x_ct}\sin{y_c t},t^{\nu_c-c}e^{-x_ct}\cos{y_c t}\right\}.\end{align*}
\begin{thm}\label{strthm}With the previous notation,
\begin{equation}\begin{aligned}\label{basis}&\hspace*{-.5em}\left\{ \widehat{L}( e^{\xi(\l_1)\l_1 t}),\dots,\widehat{L}( t^{\mu_1-1}e^{\xi(\l_1)\l_1 t}),\dots,\widehat{L}(e^{\xi(\l_r)\l_r t}),\dots,\widehat{L}( t^{\mu_r-1}e^{\xi(\l_r)\l_r t}),\right.\\ & \left. \widehat{L}(e^{\xi(z_1)x_1t}\sin{y_1 t}),\widehat{L}(e^{\xi(z_1)x_1t}\cos{y_1 t}),\dots,\widehat{L}(t^{\nu_1-1}e^{\xi(z_1)x_1t}\sin{y_1 t}),\widehat{L}(t^{\nu_1-1}e^{\xi(z_1)x_1t}\cos{y_1 t}),\right.\\ & \left.\widehat{L}(e^{\xi(z_c)x_ct}\sin{y_c t}),\widehat{L}(e^{\xi(z_c)x_ct}\cos{y_c t}),\dots,\widehat{L}(t^{\nu_c-1}e^{\xi(z_c)x_ct}\sin{y_c t}),\widehat{L}(t^{\nu_c-1}e^{\xi(z_c)x_ct}\cos{y_c t})\right\},\end{aligned}\end{equation}
where $\xi(x)=-1$ if $\til P(x)=\til Q(-x)=0$ and $\xi(x)=1$ otherwise, is a basis of the $m$-dimensional vector space of solutions of the equation $Lu=0$.
\end{thm}
\begin{proof}
Let $W$ be the real vector space  generated by~\eqref{basis}. First, by the definition of $\widehat L$, we have that, for a root $\l$ of $\widehat L L$ and $k\in\bN$,
\[\widehat{L}(t^{k}e^{\l t})=[\phi^*\til P(D)-\til Q(-D)](t^{k}e^{\l t})=(-1)^{k}\til P(\l)t^{k}e^{-\l t}-\til Q(-\l)t^{k}e^{\l t}+O(t^{k-1}e^{\l t})+O(t^{k-1}e^{-\l t}),
\]
where the notation $O(f(t))$ means that $\lim\limits_{t\to\infty}O(f(t))/f(t)$ is a real constant.
All the same,
\[\widehat{L}(t^{k}e^{-\l t})=(-1)^{k}\til P(-\l)t^{k}e^{\l t}-\til Q(\l)t^{k}e^{-\l t}+O(t^{k-1}e^{-\l t})+O(t^{k-1}e^{\l t}).
\]
Thus, since $\l$ cannot be a common root to $\til P(\l),\til P(-\l),\til 
Q(\l)$ and $\til Q(-\l)$, we have
\begin{align*}\widehat{L}(e^{\xi(\l)\l t}) & \in\<\{e^{\l t},e^{-\l t}\}\>,\\ \widehat{L}(t^{k}e^{\xi(\l)\l t}) & \in\left.\left<\{e^{\l t},\dots,t^{k}e^{\l t},e^{-\l t},\dots,t^{k}e^{-\l t}\}\>\right\backslash\<\{e^{\l t},\dots,t^{k-1}e^{\l t},e^{-\l t},\dots,t^{k-1}e^{-\l t}\}\>.\end{align*}
This means that the vectors $\widehat{L}( e^{\l_1 t}),\dots,\widehat{L}( t^{\mu_1-1}e^{\l_1 t}),\dots,\widehat{L}(e^{\l_r t}),\dots,\widehat{L}( t^{\mu_r-1}e^{\l_r t})$ are linearly independent.\par In the case of a complex root $z=x+iy$ we just have to use the following invertible linear transformation
\begin{center}\begin{tikzcd}[row sep=tiny]
\<\{e^{zt},e^{-zt}\}\> \arrow{r} & \<\{e^{zt},e^{-zt}\}\> &\\
(e^{zt},e^{-zt}) \arrow[mapsto]{r} & \frac{1}{2} \begin{pmatrix}1 & 1 \\ 1 & -1 \end{pmatrix}(e^{zt},e^{-zt}) & \hspace{-3em}=(e^{xt}\cos yt,e^{xt}\sin yt),
\end{tikzcd}
\end{center}
to deduce that the vectors~\eqref{basis} are linearly independent. Hence, $W$ is of dimension $m$.\par
By what was previously said, we know that $L|_{V}\equiv L$ maps $V$ to $V$. As such, for every $u\in W\subset V$, $u=\widehat Lv$ for some $v\in V$, so $Lu=L\widehat Lv=\widehat L L v=0$ because $V$ is the space of solutions of $\widehat L L=0$, that is, $V=\ker (\widehat L L)$. Hence, $W\subset\ker L$. Also, by construction, $W\subset \im\widehat L$, so we have that $\dim \ker L\ge m$ and $\dim \im \widehat L\ge m$.\par
 Since $\widehat{\widehat L}R=L$, we can repeat this process interchanging $L$ and $\widehat L$, so we deduce that $\dim \ker \widehat L\ge m$ and $\dim \im  L\ge\dim \ker  \widehat{\widehat L}\ge m$. Taking into account that $\dim \im  L+\dim \ker  L=\dim\im \widehat L+\dim \ker \widehat L=2m$, we conclude that $\dim \ker L=m$, which ends the proof.
\end{proof}
\section{Preliminaries: matrix functions}\label{sec3}
In the following section we will need some classical results concerning Matrix Theory which we summarize here for the convenience of the reader.

\begin{dfn}\label{matf}Let $M\in\cM_n(\bC)$ and $\l_1,\dots,\l_j$ be the (different) roots of the minimal polynomial of $M$\footnote{There exists a unique $Q\in\bC[x]$, called \emph{minimal polynomial of $M$}, such that its principal coefficient is equal to one, $Q$ divides $P(x):=|x\Id-M|$, $Q(M)=0$ and, if $R\in\bC[x]$, $R(M)=0$ and $R$ divides $Q$.}, with respective multiplicities $\mu_1,\dots,\mu_j$, and $f$ is a scalar-valued ($\bR$ or $\bC$) function such that the spectrum of $M$ is contained in the interior of the domain of $f$ and $f$ is $\mu_k-1$ times differentiable at $\l_k$ for $k=1,\dots,j$.
\par
Let $SJS^{-1}$ be the Jordan canonical form of $M$ for $S,J\in\cM_n(\bC)$, where
\[J:=\begin{pmatrix} J_{\l_1,1} & & 0 \\ & \ddots & \\ 0 & & J_{\l_j,n_j}\end{pmatrix},\]
and  $J_{\l_1,1},\dots,J_{\l_1,n_1},\dots,J_{\l_j,1},\dots,J_{\l_j,n_j}$ are de distinct Jordan blocks of  $J$.\par
Then we define the \emph{primary matrix function} $f(M)$ associated to the \emph{stem function} $f$ as
\[f(M):=S\begin{pmatrix} f(J_{\l_1,1}) & & 0 \\ & \ddots & \\ 0 & & f(J_{\l_j,n_j})\end{pmatrix}S^{-1},\]
where
\[f(J_{\l,k}):=\begin{pmatrix}
f(\l) & f'(\l) & \frac{1}{2} f''(\l) & \dots & \frac{1}{(m-1)!}f^{(m-1)}(\l) \\
0 & f(\l) & f'(\l)  & \dots & \frac{1}{(m-2)!}f^{(m-2)}(\l) \\
0& 0  & f(\l) &  \dots & \frac{1}{(m-3)!}f^{(m-3)}(\l) \\
\vdots & \vdots & \vdots  & \ddots & \vdots \\
0 & 0 & 0  & \dots & f(\l) \\
\end{pmatrix},\]
assuming $J_{\l,k}\in\cM_m(\bC)$.
\end{dfn}
\begin{lem}[{\cite[Theorem 6.2.9(b)]{HJ}}]Let $M\in\cM_n(\bC)$. $f(M)$ is independent of the particular Jordan canonical form used to represent $M$.
\end{lem}
\begin{thm}[{\cite[Theorem 6.2.8]{HJ}}]
Let $f$ be a scalar-valued analytic function with a power series representation $f(t)=\sum_{k=0}^\infty\a_kt^k$ that has radius of convergence $r>0$. If $M\in\cM_n(\bC)$ is such that $\rho(M)<r$, then the matrix power series $\sum_{k=0}^\infty\a_kM^k$ converges absolutely and is equal to $f(M)$.
\end{thm}
\begin{thm}[{\cite[Theorem 6.2.9]{HJ}}]\label{gft} Let $M\in\cM_n(\bC)$ with minimal polynomial $P\in\bC[x]$ and let $\l_1,\dots,\l_j$ be the (different) roots of $P$ with respective multiplicities $\mu_1,\dots,\mu_j$. Let $f$ and $g$ be scalar-valued ($\bR$ or $\bC$) functions such that the spectrum of $M$ is contained in the interior of their domains, being $f$ and $g$ $\mu_k-1$ times differentiable at $\l_k$ for $k=1,\dots,j$. Then
\begin{enumerate}
\item There exists $P\in\bC[x]$ of degree less or equal than $n-1$ such that $f(M)=P(M)$. $P$ may be taken to be any polynomial that interpolates $f$ and its derivatives at the roots of the minimal polynomial of $M$ (according to multiplicity).
\item $g(M)=f(M)$ if and only if $g^{(r)}(\l_k)=f^{(r)}(\l_k)$ for $r=0,\dots,\mu_{k}-1$, $k=0,\dots, j$.
\item The eigenvalues of $f(M)$ are $f(\l_1),\dots, f(\l_j)$, taking into account the multiplicity. The shape of the Jordan boxes of $M$ is preserved through $f$, but changing the eigenvalues to those of $f(M)$.
\end{enumerate}
\end{thm}
\par The following result provides a square root of an invertible matrix with desirable commutativity properties.
\begin{pro}\label{prosq} Let $M,N\in\cM_n(\bC)$ such that $M$ is invertible and $M$ and $N$ commute. Then, if $\l_1,\dots,\l_j$ are roots of the minimal polynomial of $M$ with respective multiplicities $\mu_1,\dots,\mu_j$, there exist $P\in\bC[x]$ such that

\begin{enumerate}
\item $P(M)N=NP(M)$,
\item $P(M)^2=M$.
\end{enumerate}
\end{pro}
\begin{proof}1. This is straightforward from the facts that $P(M)$ is a polynomial on $M$ and that $M$ and $N$ commute.\par
2. Define $h(t)=\sqrt{t}$ as the principal branch of the square root in $\bC$. $h$ is differentiable in $\bC\backslash\{0\}$, so it is clear that $h$ is well defined and $\mu_k-1$ differentiable at $\l_k\ne0$, for $k=0,\dots,j$.

Let $P$ be the interpolating polynomial defined in Theorem~\ref{gft}.1 for the function $h$, that is, the one such that $h^{(r)}(\l_k)=P^{(r)}(\l_k)$ for $r=0,\dots,\mu_{k}-1$, $k=0,\dots, j$. Now we apply Theorem~\ref{gft}.2 to the functions $f(t)=P(t)^2$ and $g(t)=t$.
We have that \[f^{(r)}(\l_k)=(P^2)^{(r)}(\l_k)=(h^2)^{(r)}(\l_k)=g^{(r)}(\l_k),\] for $r=0,\dots,\mu_{k}-1$, $k=0,\dots, j$, which ends the result.

\end{proof}
\begin{rem}The square root provided in Proposition \ref{prosq} is invertible since $M$ is invertible. This is because $|M|=|P(M)^2|=|P(M)|^2\ne 0$.
\end{rem}
We can prove in an analogous fashion, taking the principal branch of the logarithm in $\bC$, the following proposition.
\begin{pro}\label{prosq2} Let $M,N\in\cM_n(\bC)$ be such that $M$ is invertible and $M$ and $N$ commute. Then, if $\l_1,\dots,\l_j$ are the roots of the minimal polynomial of $M$ with respective multiplicities $\mu_1,\dots,\mu_j$, there exists $P\in\bC[x]$ such that
\begin{enumerate}
\item $P(M)N=NP(M)$,
\item $e^{P(M)}=M$.
\end{enumerate}
\end{pro}

Now we state some results concerning block matrices. Consider $M_1,\dots,M_4\in\cM_n(\bC)$ and consider the block matrix
  \[\bm M=\(\begin{array}{c|c}
   M_1 & M_2 \\ \hline
   M_3 & M_4
   \end{array}\).\]
 \begin{lem}\label{lemdetmat}\ 
 
 \begin{enumerate}
\item If $M_1$ and $M_3$ commute, $|\bm M|=|M_1M_4-M_3M_2|$.
\item If $M_3=-M_2$ and $M_4=-M_1$, $|\bm M|=|M_1+M_2||M_2-M_1|$.
\end{enumerate}
 \end{lem}

\begin{proof} Statement 1  appears in \cite[Theorem 3]{Silvester}. In order to prove 2 we use the basic properties of the determinant, namely, that if we add to some row (column) a linear combination of the rest, the determinant does not vary. Hence, we have that
\[\left|\begin{array}{c|c}
M_1 & M_2 \\ \hline
-M_2 & -M_1
\end{array}\right|=\left|\begin{array}{c|c}
M_1+M_2 & M_2+M_1 \\ \hline
-M_2 & -M_1
\end{array}\right|=\left|\begin{array}{c|c}
M_1+M_2 &0\\ \hline
-M_2 & -M_1+M_2
\end{array}\right|=|M_1+M_2||M_2-M_1|.\]
\end{proof}
\section{Systems of linear equations with reflection}\label{sec4}
In this section we will consider the homogeneous system of linear equations
\begin{equation}\label{hlsystem}Hu(t):=Fu'(t)+Gu'(-t)+A u(t)+Bu(-t)=0, t\in\bR,
\end{equation}
where $n\in\bN$, $A,B,F,G\in\cM_n(\bR)$ and $u:\bR\to\bR^n$. We will prove that a fundamental matrix for problem~\eqref{hlsystem} exists.
\begin{dfn}
$X\in\cC\(\bR,\cM_n(\bR)\)$ is a \textit{fundamental matrix} of problem~\eqref{hlsystem} if it satisfies the following properties:
\begin{itemize}
\item[(H1)]$X$ is differentiable,
\item[(H2)]$X$ satisfies equation~\eqref{hlsystem}, that is
\begin{equation*}
FX'(t)+GX'(-t)+A X(t)+BX(-t)=0, t\in\bR,
\end{equation*}
\item[(H3)] $X(0)$ is invertible,
\item[(H4)] the columns of $X$ generate a basis of the space of solutions of~\eqref{hlsystem}.
\end{itemize}\par
\end{dfn}
In order to prove the existence of a fundamental matrix we will have to find a reduced equation for problem~\eqref{hlsystem} --cf. Theorem~\ref{Rop2}-- and  rewrite Theorem~\ref{strthm} in terms of a system of equations.
\subsection{Reducing the system}

Consider the notation $\ol u(t):=u(-t),\ t\in\bR$ for every function $u:\bR\to\bR^n$\footnote{There will be no mention to the complex conjugate (an involution as well) in this paper, so this notation will not cause any confusion.}. This way, equation~\eqref{hlsystem} can be expressed as $Fu'+G\ol{ u'}+Au+B\ol u=0$. Also, $H=FD+G\phi^*D+A+B\phi^*$.\par  Now Define $v=\ol u$. Then, we can rewrite equation \eqref{hlsystem} as
\[Fu'-Gv'+Au+Bv=0.\]
Evaluating this expression in $-t$ we get
\[-Fv'+Gu'+Av+Bu=0.\]
Expressing these two equations as a system, we have
\begin{equation}\label{equv}\(\begin{array}{c|c}
F & -G\\ \hline
G & -F
\end{array}\)\(\begin{array}{c}
u' \\ \hline v'
\end{array}\)+\(\begin{array}{c|c}
A & B\\ \hline
B & A
\end{array}\)\(\begin{array}{c}
u \\ \hline v
\end{array}\)=0.\end{equation}
\begin{lem}\label{lemuos} If $F+G$ and $F-G$ are invertible, equation \eqref{hlsystem} coupled with the condition $u(0)=c$ for some $c\in\bR^n$ has, at most, one solution.
\end{lem}
\begin{proof} By Lemma \ref{lemdetmat}.2 we have that
\[\left|\begin{array}{c|c}
F & -G\\ \hline
G & -F
\end{array}\right|=|F-G||-G-F|=(-1)^n|F+G||F-G|\ne0,\]
so equation \eqref{equv} can be rewritten as
\begin{equation}\label{equv2}\(\begin{array}{c}
u' \\ \hline v'
\end{array}\)=-\(\begin{array}{c|c}
F & -G\\ \hline
G & -F
\end{array}\)^{-1}\(\begin{array}{c|c}
A & B\\ \hline
B & A
\end{array}\)\(\begin{array}{c}
u \\ \hline v
\end{array}\).\end{equation}
 If $u(0)=c$, we have that $v(0)=c$. On the other hand, equation \eqref{equv2} coupled with these initial conditions has exactly one solution due to the theory of ordinary differential systems.
\end{proof}
\begin{cor}\label{fmu}
If $F+G$ and $F-G$ are invertible and $X$, $Y\in\cC\(\bR,\cM_n(\bR)\)$ are matrices satisfying (H1)--(H3), then there exists an invertible matrix $\Lambda\in\cM_n(\bR)$ such that $Y=X\Lambda$. Furthermore, $X$ and $Y$ satisfy (H4).
\end{cor}
\begin{proof}Let $c\in\bR^n$ be any vector and define
\[u(t)=(Y(t)Y(0)^{-1}-X(t)X(0)^{-1})c,\ t\in\bR.\] By construction, $u$ is a solution of problem
\begin{equation}\label{eqic} Hu=0,\ u(0)=0,
\end{equation}
 so, by Lemma \ref{lemuos}, $u=0$. Since $v$ was fixed arbitrarily, \[Y(t)Y(0)^{-1}=X(t)X(0)^{-1},\] so the statement of the Corollary holds for $\Lambda=X(0)^{-1}Y(0)$.
\par
Now, take a solution $v$ of problem \eqref{eqic} and consider the problem \eqref{eqic} coupled with the condition $u(0)=v(0)$. By Lemma \ref{lemuos}, $v$ is the unique solution to this problem. On the other hand, $X(t)X(0)^{-1}v(0)$ is a solution of the problem as well, so $v(t)=X(t)X(0)^{-1}v(0)$ and therefore $X$ generates all of the solutions of problem \eqref{eqic}. The same happens for $Y$.
\end{proof}
\begin{rem}\label{remsufc}Corollary \ref{fmu} establishes a sufficient condition ($F+G$ and $F-G$ invertible) for matrices satisfying (H1)--(H3) to be fundamental matrices of problem \eqref{eqic}.
\end{rem}
Lemma \ref{lemuos} establishes an upper bound for the number of solutions of our equation. But, is there a solution at all? Now we try to answer this question.

\subsection{Explicit computation of the fundamental matrix}
In this last part of the Section we move towards a more direct approach to the fundamental matrix of problem~\eqref{hlsystem}. Furthermore, we find a simpler explicit expression for it under certain assumptions.

\begin{thm}\label{thmexpfm}
Assume $F-G$ and $F+G$ are invertible. Then
\begin{equation}\label{Xseries}X(t): =     \sum_{k=0}^\infty\frac{E^k t^{2k}}{(2k)!}  -(F+G)^{-1}(A+B)\sum_{k=0}^\infty\frac{E^k t^{2k+1}}{(2k+1)!},\end{equation}
where $E=(F-G)^{-1}(A-B)(F+G)^{-1}(A+B)$, is a fundamental matrix of problem \eqref{hlsystem}. If we further assume $A-B$ and $A+B$ are invertible, then $E$ is invertible and we can consider $\Omega$ to be the root of $E$ constructed in Proposition \ref{prosq}. Then,
\begin{equation}\label{fme}X(t)=\cosh \Omega t -(F+G)^{-1}(A+B)\Omega^{-1}\sinh\Omega t.\end{equation}
\end{thm}
\begin{proof}
 We observe that the series that defines $X$ is uniformly and absolutely convergent by the ratio test and, therefore, $X$ is an analytic function. Furthermore, $X(0)=\Id$ is invertible. \par
 Also, the relationship between formulas \eqref{Xseries} and \eqref{fme} is clear (assuming the complementary hypotheses): 
 \begin{align*}X(t)= &  \sum_{k=0}^\infty\frac{E^k t^{2k}}{(2k)!}  -(F+G)^{-1}(A+B)\sum_{k=0}^\infty\frac{E^k t^{2k+1}}{(2k+1)!}= \sum_{k=0}^\infty\frac{(\Omega t)^{2k}}{(2k)!}-(F+G)^{-1}(A+B)\Omega^{-1}\sum_{k=0}^\infty\frac{(\Omega t)^{2k+1}}{(2k+1)!} . \end{align*}
 Observe that, although $\Omega$ might not be a matrix with real coefficients, $X$ is a real matrix.\par
  We check now that $X(t)$ satisfies equation \eqref{hlsystem}. Define, for convenience,
  \[S_1(t):=\sum_{k=0}^\infty\frac{E^k t^{2k}}{(2k)!} ,\quad S_2(t):=\sum_{k=0}^\infty\frac{E^k t^{2k+1}}{(2k+1)!}.\]
  Observe that $S_1$ is even and $S_2$ is odd. Then,
\begin{alignat*}{2}X(t) &=     S_1(t)  -(F+G)^{-1}(A+B)S_2(t),\quad X'(t) &&=    E\,S_2(t)-(F+G)^{-1}(A+B)S_1(t), \\ X(-t) &=     S_1(t)  +(F+G)^{-1}(A+B)S_2(t),\quad X'(-t) &&=   -E\,S_2(t)-(F+G)^{-1}(A+B)S_1(t).\end{alignat*}

  \begin{align*} & FX'(t)+GX'(-t)+A X(t)+BX(-t)\\ = & 
F(E\,S_2(t)-(F+G)^{-1}(A+B)S_1(t))  +G( -E\,S_2(t)-(F+G)^{-1}(A+B)S_1(t))\\ & +A(S_1(t)  -(F+G)^{-1}(A+B)S_2(t))  +B(S_1(t)  +(F+G)^{-1}(A+B)S_2(t))
  \\ = & [-F(F+G)^{-1}(A+B)-G(F+G)^{-1}(A+B)+A+B]S_1(t) \\ &+[F\,E-G\,E-A(F+G)^{-1}(A+B)+B(F+G)^{-1}(A+B)]S_2(t)=0.
  \end{align*}
Hence $X$ satisfies (H1)--(H3) and, by Remark \ref{remsufc}, $X$ is a fundamental matrix.
\end{proof}
\begin{rem} Observe an important difference between the case of systems of ordinary differential equations and the case with reflection. While in the first case we know that fundamental matrices are invertible for every $t\in\bR$, in the case of the reflection, expression \eqref{Xseries} shows that, although $X$ is invertible in a neighborhood of $0$, it might not be so for other values of $t$. This goes in the line of \cite[Lemma 2.4]{Toj3}.
\end{rem}
\begin{rem} Observe that in the scalar case of $n=1$ with $F=1$, $G=0$, Theorem \ref{thmexpfm} provides the same expression for $X$ as the one computed in \cite[Section 2.2, Case (C2)]{Toj3}.
\end{rem}
\begin{rem} In the proof of Theorem \ref{thmexpfm} it is not actually needed the square root to be the one constructed in Proposition \ref{prosq}, it could be any other square root, but, as we will see in the results to come, to choose this one is specially useful because of its commutativity properties.
\end{rem}
\begin{exa}Consider the system of equations
\begin{equation*}\begin{aligned} u'(t) & =-v(t), \\ v'(-t), & =-u(-t),\end{aligned}
\end{equation*}
which expressed in matrix form is
\[\begin{pmatrix} 1 & 0 \\ 0 & 0
\end{pmatrix}\begin{pmatrix} u'(t) \\ v'(t) \end{pmatrix}+\begin{pmatrix} 0 & 0 \\ 0 & 1
\end{pmatrix}\begin{pmatrix} u'(-t) \\ v'(-t) \end{pmatrix}+\begin{pmatrix} 0 & 1 \\ 0 & 0
\end{pmatrix}\begin{pmatrix} u(t) \\ v(t) \end{pmatrix}+\begin{pmatrix} 0 & 0 \\ 1 & 0
\end{pmatrix}\begin{pmatrix} u(-t) \\ v(-t) \end{pmatrix}=0.\]
Hence,
\[(F-G)^{-1}(A-B)(F+G)^{-1}(A+B)=\begin{pmatrix} 1 & 0 \\ 0 & -1
\end{pmatrix}\begin{pmatrix} 0 & 1 \\ -1 & 0
\end{pmatrix}\begin{pmatrix} 1 & 0 \\ 0 & 1
\end{pmatrix}\begin{pmatrix} 0 & 1 \\ 1 & 0
\end{pmatrix}=\Id,\]
and, therefore,
\[X(t)=\begin{pmatrix} \cosh t & 0 \\ 0 & \cosh t
\end{pmatrix} -\begin{pmatrix} 0 & 1 \\ 1 & 0
\end{pmatrix}\begin{pmatrix} \sinh t & 0 \\ 0 & \sinh t
\end{pmatrix}=\begin{pmatrix} \cosh t & -\sinh t \\ -\sinh t & \cosh t
\end{pmatrix}.\]
$X(t)$ is invertible for $t\in\bR$ since $\cosh^2 t-\sinh^2 t=1$ for every $t\in\bR$.
\end{exa}
\subsubsection{Computation of the matrix $E$ for a differential equation}
Let us consider equation $Lu=0$ --where $L$ is defined as in \eqref{Lop}. We rewrite $Lu=0$ in terms of its coefficients 
\begin{displaymath}Lu(t)=\sum_{k=0}^na_ku^{(k)}(t)+\sum_{k=0}^nb_ku^{(k)}(-t)=0,\ t\in I,\end{displaymath}
and now make the change of variables $x_0=u$, $x_1=x_0',\dots,x_{n}=x_{n-1}'$.  We have that $Lu=0$ is equivalent to the system
\begin{equation}\label{fstsys}\begin{aligned}
x'_k -x_{k+1} & =0,\ k=0,\dots,n-2,\\
a_{n}x_{n-1}'+b_{n}\ol{x_{n-1}'}+\sum_{k=0}^{n-1}a_kx_k+\sum_{k=0}^{n-1}b_k\ol{x_k}& =0.
\end{aligned}\end{equation}
Hence, taking $x=(x_0,\dots,x_{n-1})$, and $A,B,F,G\in\cM_n(\bR)$ such that
\[F=\(\begin{array}{c|c}
\Id &\bm 0 \\ \hline
\bm 0 & a_{n}
\end{array}\),\ G=\(\begin{array}{c|c}
\bm 0 & \bm 0 \\ \hline
\bm 0 & b_{n}
\end{array}\)
 ,\ A=\begin{pmatrix}
0 & -1 & 0 &  \cdots & 0 & 0 \\ 
0 & 0 & -1 &  \cdots & 0 & 0 \\ 
\vdots &  \vdots & \vdots & \ddots & \vdots & \vdots \\ 
0 & 0 & 0 &  \cdots & -1 & 0 \\ 
0 & 0 & 0 &  \cdots & 0 & -1 \\ 
a_0 & a_1 & a_2 &  \cdots & a_{n-2} & a_{n-1}
\end{pmatrix},\ B=\begin{pmatrix}
0 & \cdots & 0 \\ 
\vdots & \ddots & \vdots \\ 
0 & \cdots & 0 \\ 
b_0 & \cdots & b_{n-1}
\end{pmatrix}, 
\]
where $\bm 0$ denotes a zero matrix, we have that the system~\eqref{fstsys} can be expressed in the form~\eqref{hlsystem}.
Hence, if we assume $a_n^2\ne b_n^2$,
\[(F-G)^{-1}  =\(\begin{array}{c|c}
\Id &\bm 0 \\ \hline
\bm 0 &( a_{n}-b_n)^{-1}
\end{array}\)
 ,\ (F+G)^{-1}=\(\begin{array}{c|c}
 \Id &\bm 0 \\ \hline
 \bm 0 &( a_{n}+b_n)^{-1}
 \end{array}\),
\]
and $(F-G)^{-1}(A-B)(F+G)^{-1}(A+B)=$
 \begin{align*}  & \begin{pmatrix}
0 & -1 & 0 &  \cdots & 0 & 0 \\ 
0 & 0 & -1 &  \cdots & 0 & 0 \\ 
\vdots &  \vdots & \vdots & \ddots & \vdots & \vdots \\ 
0 & 0 & 0 &  \cdots & -1 & 0 \\ 
0 & 0 & 0 &  \cdots & 0 & -1 \\ 
\frac{a_0-b_0}{a_n-b_n} & \frac{a_1-b_1}{a_n-b_n} & \frac{a_2-b_2}{a_n-b_n} &  \cdots & \frac{a_{n-2}-b_{n-2}}{a_n-b_n} &  \frac{a_{n-1}-b_{n-1}}{a_n-b_n}
\end{pmatrix}\begin{pmatrix}
0 & -1 & 0 &  \cdots & 0 & 0 \\ 
0 & 0 & -1 &  \cdots & 0 & 0 \\ 
\vdots &  \vdots & \vdots & \ddots & \vdots & \vdots \\ 
0 & 0 & 0 &  \cdots & -1 & 0 \\ 
0 & 0 & 0 &  \cdots & 0 & -1 \\ 
\frac{a_0+b_0}{a_n+b_n} & \frac{a_1+b_1}{a_n+b_n} & \frac{a_2+b_2}{a_n+b_n} &  \cdots & \frac{a_{n-2}+b_{n-2}}{a_n+b_n} &  \frac{a_{n-1}+b_{n-1}}{a_n+b_n} 
\end{pmatrix} \\ = & 
  \begin{pmatrix}
0 & 0 & 1 &  \cdots &  0 \\ 
\vdots &  \vdots & \vdots & \ddots &  \vdots \\ 
0 & 0 & 0 &  \cdots & 1 \\ 
-\frac{a_0+b_0}{a_n+b_n} & -\frac{a_1+b_1}{a_n+b_n} & -\frac{a_2+b_2}{a_n+b_n} &  \cdots & - \frac{a_{n-1}+b_{n-1}}{a_n+b_n} \\
\frac{(a_{n-1}-b_{n-1})(a_0+b_0)}{a_n^2-b_n^2} & \frac{(a_{n-1}-b_{n-1})(a_1+b_1)}{a_n^2-b_n^2}-\frac{a_0+b_0}{a_n+b_n} & \frac{(a_{n-1}-b_{n-1})(a_2+b_2)}{a_n^2-b_n^2}-\frac{a_1+b_1}{a_n+b_n} &  \cdots &   \frac{a_{n-1}^2-b_{n-1}^2}{a_n^2-b_n^2}- \frac{a_{n-2}+b_{n-2}}{a_n+b_n}
\end{pmatrix}.
\end{align*}
This expression is too convoluted to compute the square root in a general way, but we can study some simpler settings with further assumptions.
\begin{exa}
Take $n=4$ and assume $a_j=-b_j$, $j=1,2$; $a_3=b_3=0$; $\c=-(a_0+b_0)/(a_4+b_4)>0$, that is,
\begin{equation}\label{exaom}Lu=a_4u^{(4)}(t)+b_4u^{(4)}(-t)+a_2[u''(t)-u''(-t)]+a_1[u'(t)-u'(-t)]+a_0u(t)+b_0u(-t)=0.\end{equation}
Then
\[(F-G)^{-1}(A-B)(F+G)^{-1}(A+B)= \begin{pmatrix}
0 & 0 & 1   &  0 \\ 
0 & 0 & 0  & 1 \\ 
\c &0 & 0    & 0 \\
0 & \c & 0 &   0
\end{pmatrix},\]
and 
\[\Omega=\left(
\begin{array}{cccc}
 \left(\frac{1}{2}+\frac{i}{2}\right) \sqrt[4]{\gamma } & 0 & \frac{\frac{1}{2}-\frac{i}{2}}{\sqrt[4]{\gamma }} & 0 \\
 0 & \left(\frac{1}{2}+\frac{i}{2}\right) \sqrt[4]{\gamma } & 0 & \frac{\frac{1}{2}-\frac{i}{2}}{\sqrt[4]{\gamma }} \\
 \left(\frac{1}{2}-\frac{i}{2}\right) \gamma ^{3/4} & 0 & \left(\frac{1}{2}+\frac{i}{2}\right) \sqrt[4]{\gamma } & 0 \\
 0 & \left(\frac{1}{2}-\frac{i}{2}\right) \gamma ^{3/4} & 0 & \left(\frac{1}{2}+\frac{i}{2}\right) \sqrt[4]{\gamma } \\
\end{array}
\right).\]
All the same,
\[X(t)=\left(
\begin{smallmatrix}
 \frac{1}{2} \left(\cos \left(t \sqrt[4]{\gamma }\right)+\cosh \left(t \sqrt[4]{\gamma }\right)\right) & \frac{\sin \left(t \sqrt[4]{\gamma }\right)+\sinh \left(t \sqrt[4]{\gamma }\right)}{2 \sqrt[4]{\gamma }} & \frac{\cosh \left(t \sqrt[4]{\gamma }\right)-\cos \left(t \sqrt[4]{\gamma }\right)}{2 \sqrt{\gamma }} & \frac{\sinh \left(t \sqrt[4]{\gamma }\right)-\sin \left(t \sqrt[4]{\gamma }\right)}{2 \gamma ^{3/4}} \\
 \frac{1}{2} \sqrt[4]{\gamma } \left(\sinh \left(t \sqrt[4]{\gamma }\right)-\sin \left(t \sqrt[4]{\gamma }\right)\right) & \frac{1}{2} \left(\cos \left(t \sqrt[4]{\gamma }\right)+\cosh \left(t \sqrt[4]{\gamma }\right)\right) & \frac{\sin \left(t \sqrt[4]{\gamma }\right)+\sinh \left(t \sqrt[4]{\gamma }\right)}{2 \sqrt[4]{\gamma }} & \frac{\cosh \left(t \sqrt[4]{\gamma }\right)-\cos \left(t \sqrt[4]{\gamma }\right)}{2 \sqrt{\gamma }} \\
 \frac{1}{2} \sqrt{\gamma } \left(\cosh \left(t \sqrt[4]{\gamma }\right)-\cos \left(t \sqrt[4]{\gamma }\right)\right) & \frac{1}{2} \sqrt[4]{\gamma } \left(\sinh \left(t \sqrt[4]{\gamma }\right)-\sin \left(t \sqrt[4]{\gamma }\right)\right) & \frac{1}{2} \left(\cos \left(t \sqrt[4]{\gamma }\right)+\cosh \left(t \sqrt[4]{\gamma }\right)\right) & \frac{\sin \left(t \sqrt[4]{\gamma }\right)+\sinh \left(t \sqrt[4]{\gamma }\right)}{2 \sqrt[4]{\gamma }} \\
 \frac{1}{2} \gamma ^{3/4} \left(\sin \left(t \sqrt[4]{\gamma }\right)+\sinh \left(t \sqrt[4]{\gamma }\right)\right) & \frac{1}{2} \sqrt{\gamma } \left(\cosh \left(t \sqrt[4]{\gamma }\right)-\cos \left(t \sqrt[4]{\gamma }\right)\right) & \frac{1}{2} \sqrt[4]{\gamma } \left(\sinh \left(t \sqrt[4]{\gamma }\right)-\sin \left(t \sqrt[4]{\gamma }\right)\right) & \frac{1}{2} \left(\cos \left(t \sqrt[4]{\gamma }\right)+\cosh \left(t \sqrt[4]{\gamma }\right)\right) \\
\end{smallmatrix}
\right).\]
Hence, any solution equation \eqref{exaom} is of the form
\[u(t)= c_1 \cos\sqrt[4]{\gamma } t+c_2\sin\sqrt[4]{\gamma } t+c_3 \cosh\sqrt[4]{\gamma } t+c_4\sinh\sqrt[4]{\gamma } t,\]
where $c_1,\dots,c_4\in\bR$ are arbitrary constants.
\end{exa}
Finally, we introduce a direct Corollary of Theorem \ref{thmexpfm}.
\begin{cor}\label{corfmu} If $F-G$, $F+G$, $A-B$ and $A+B$ are invertible, every matrix $Y$ satisfying (H1)--(H3) for problem \eqref{hlsystem} is of the form $Y=X\Lambda$ for some invertible $\Lambda=Y(0)\in\cM(\bR_n)$ where $X$ is defined as in \eqref{fme}.
\end{cor}
\begin{proof} The result is straightforward from Corollary \ref{fmu} and Theorem \ref{thmexpfm}. Since $X(0)=\Id$, $Y(0)=\Lambda$.
\end{proof}
\subsection{Rewriting of the fundamental matrix using the PAF}
The next results study those values of $t$ for which $X(t)$, given by expression \eqref{fme}, is singular. The following theorem is inspired in a result of \cite{Toj4} called the (Hyperbolic) Phasor Addition Formula which we state now.
\begin{lem}[\textsc{Hyperbolic Phasor Addition Formula}, \cite{Toj4}]
Let $ a$, $ b$, $\omega\in \mathbb R$. Then
\begin{equation*} a \cosh  \omega t +  b\sinh \omega= \begin{cases} \sqrt{| a^2- b^2|} \cosh\left(\frac{1}{2}\ln\left|\frac{ a+ b}{ a- b}\right|+ \omega \right), & a>| b|,
\\ -\sqrt{| a^2- b^2|}  \cosh\left(\frac{1}{2}\ln\left|\frac{ a+ b}{ a- b}\right|+ \omega \right), & -  a>| b|,
\\
\sqrt{| a^2- b^2|} \sinh\left(\frac{1}{2}\ln\left|\frac{ a+ b}{ a- b}\right|+ \omega \right), &   b>| a|,
\\
-\sqrt{| a^2- b^2|} \sinh\left(\frac{1}{2}\ln\left|\frac{ a+ b}{ a- b}\right|+ \omega \right), &  - b>| a|,
\\
 a\,e^ {\omega }, &  a= b,\\
 a\,e^{- \omega }, & a=- b.\end{cases}\end{equation*}
\end{lem}
\begin{thm}[\textsc{Matrix Phasor Addition Formula}]\label{mpaf} Let $M,N,U\in\cM_n(\bC)$ such that $M+N$ and $M-N$ are invertible and $M$, $N$ and $U$ commute. Then, if $M_0, N_0$ are, respectively, the square roots of $M+N$ and $M-N$, according to Proposition~\ref{prosq}, and $U_0$ is the logarithm of $N_0^{-1}M_0$, according to Proposition~\ref{prosq2}, the following identity holds:
\begin{equation*}\begin{aligned}M\cosh U+N\sinh U & =M_0N_0\cosh(U_0+U)\\ & =\sqrt{(M+N)(M-N)}\cosh\left(\ln \left[\left(\sqrt{M-N}\right)^{-1}\sqrt{M+N}\right]+U\right).\end{aligned}\end{equation*}
\end{thm}
\begin{proof} Since $M$ and $N$ commute, so do $M+N$ and $M-N$. Also, since  $M+N$ and $M-N$ commute, and $M+N$ and $M_0$ and $M+N$ and $N_0$ commute (because they are their respective square roots), we deduce, from Proposition~\ref{prosq}, that $M_0$ and $N_0$ commute. Let $\c=e^U$. $U$ commutes with $M$ and $N$, so $\c$ commutes with $M_0$ and $N_0$. Then
\begin{align*} & M\cosh U+N\sinh U=M\frac{1}{2}\(\c+\c^{-1}\)+N\frac{1}{2}\(\c-\c^{-1}\)=\frac{1}{2}\left[\(M+N\)\c+\(M-N\)\c^{-1}\right]\\= & \frac{1}{2}\(M_0^2\c+N_0^2\c^{-1}\)=  \frac{1}{2}M_0\(M_0\c+M_0^{-1}N_0^2\c^{-1}\)=  \frac{1}{2}M_0\(M_0\c+N_0M_0^{-1}N_0\c^{-1}\) \\= & \frac{1}{2}M_0N_0\(N_0^{-1}M_0\c+M_0^{-1}N_0\c^{-1}\)=\frac{1}{2}M_0N_0\left[N_0^{-1}M_0\c+\(N_0^{-1}M_0\c\)^{-1}\right]=\frac{1}{2}M_0N_0\(e^{U_0+U}+e^{-(U_0+U)}\)\\= & M_0N_0\cosh(U_0+U).
\end{align*}
\end{proof}
Now we can state the following result as a direct consequence of the Phasor Addition Formula.
\begin{lem}\label{lempha}If $A+B$, $A-B$, $F+G$, $F-G$, $\Id+(F+G)^{-1}(A+B)$ and $\Id-(F+G)^{-1}(A+B)$ are invertible, and $(F+G)^{-1}(A+B)$ and $\Omega$ commute, then
\begin{equation*}\begin{aligned}X(t)  = & \cosh \Omega t -(F+G)^{-1}(A+B)\Omega^{-1}\sinh\Omega t\\  = & \sqrt{\left[\Id-(F+G)^{-1}(A+B)\right]\left[\Id+(F+G)^{-1}(A+B)\right]}\\ & \cdot \cosh\left(\ln \left[\left(\sqrt{\Id+(F+G)^{-1}(A+B)}\right)^{-1}\sqrt{\Id-(F+G)^{-1}(A+B)}\right]+\Omega t\right).\end{aligned}\end{equation*}
\end{lem}
\begin{cor}Under the conditions of Lemma \ref{lempha}, $X(t)$ is singular for every $t$ satisfying
\[\left|\Lambda e^{\Omega t}\pm i\Id\right|=0,\]
where
\[\Lambda=\left(\sqrt{\Id+(F+G)^{-1}(A+B)}\)^{-1}\sqrt{\Id-(F+G)^{-1}(A+B)}.\]
\end{cor}
\begin{proof} From Theorem \ref{gft}.3, we know that, for any matrix $M$, $\sp(\cosh M)=\cosh(\sp(M))$ (where $\sp(M)$ denotes the spectrum of $M$), so, in order to see whether $M$ is singular or not we have to check if the eigenvalues of $M$ are in the set of zeros of $\cosh$, that is $Z_1=\left\{(k+1/2)\pi i\ :\ k\in\bZ\right\}$.\par
Now,
\begin{align*} &\cosh\left(\ln \left[\left(\sqrt{\Id+(F+G)^{-1}(A+B)}\right)^{-1}\sqrt{\Id-(F+G)^{-1}(A+B)}\right]+\Omega t\right)\\  = &\cosh\left(\ln \left[\left(\sqrt{\Id+(F+G)^{-1}(A+B)}\right)^{-1}\sqrt{\Id-(F+G)^{-1}(A+B)}e^{\Omega t}\right]\right).
\end{align*}
Thus, in order for the logarithm of $z\in\bC$ to be in $Z_1$, we need $z=\pm i$. Hence, we have to solve the equation
\[\left|\Lambda e^{\Omega t}\pm i\Id\right|=0.\]
\end{proof}
\begin{exa}\label{exasen}Consider the system of equations
\begin{equation*}\begin{aligned} u'(t) & =\a\, u(-t)+ \b\, v(-t) \\ v'(t), & =\c\, u(-t)+ \d\, v(-t),\end{aligned}
\end{equation*}
which expressed in matrix form is
\[F\begin{pmatrix} u'(t) \\ v'(t) \end{pmatrix}+B\begin{pmatrix} u(-t) \\ v(-t) \end{pmatrix}=\begin{pmatrix} 1 & 0 \\ 0 & 1
\end{pmatrix}\begin{pmatrix} u'(t) \\ v'(t) \end{pmatrix}+\begin{pmatrix} -\a & -\b \\ -\c & -\d
\end{pmatrix}\begin{pmatrix} u(-t) \\ v(-t) \end{pmatrix}=0.\]
Hence, $\Omega^2=-B^2=(iB)^2$, so we take $\Omega=iB$. Using expression \eqref{Xseries} we get
\begin{equation*}X(t)  =  \cos Bt-\sin Bt.\end{equation*}

On the other hand, assuming $B$, $\Id+B$ and $\Id-B$ are invertible, using Lemma \ref{lempha}, we know that
\[X(t)  = \sqrt{\Id-B^2} \cosh\left(\ln \left[\left(\sqrt{\Id+B}\right)^{-1}\sqrt{\Id-B}\right]+iB t\right).\]
In order to do some explicit computations, we simplify the problem. Let us assume $\a=\d=0$, $\b,\c>0$. Then, using expression \eqref{Xseries}, we get
\[X(t)=
\begin{pmatrix}
 \cos \left( \sqrt{\beta\gamma }t\right) & \sqrt\frac{\b}{\c} \sin \left(\sqrt{\beta \gamma } t\right) \\
 \sqrt\frac{\c}{\b} \sin \left(\sqrt{\beta \gamma } t\right)  & \cos \left( \sqrt{\beta \gamma }t\right)
\end{pmatrix}.\]
We have that \[|X(t)|=\cos^2 \left( \sqrt{\beta \gamma }t\)-\sin^2 \left( \sqrt{\beta \gamma }t\)=\cos\(2\sqrt{\b\c}t\).\]
Hence, $X(t)$ is singular if and only if $t\in\{(k\pi\pm\pi/4)/\sqrt{\b\c}\ :\ k\in\bZ\}$. In particular, $X$ is regular in $\(-\pi/(4\sqrt{\b\c}),\pi/(4\sqrt{\b\c})\)$.
\end{exa}

\section{The method of variation of parameters for the reflection}\label{sec5}
Now that we have proved the existence of a fundamental matrix of problem~\eqref{hlsystem}, we attempt to develop an analog of the well known method of variation of parameters for the case of the reflection.
\subsection{A first attempt: the classical method}
Now we have a problem of the kind
\begin{equation}\label{nhp}u'=Au+B\ol u+\gamma,\end{equation}
where $A,\,B\in\cM_n(\bR)$ and $\gamma\in\bR^n$.

We can consider, as done in the previous section, a fundamental matrix $X$ of  the associated homogeneous problem
\[u'=Au+B\ol u.\]
That is, $X$ satisfies
\begin{equation}\label{pfm}X'=AX+B\ol X.\end{equation}
If we undertake the same approach as in the case of ordinary differential equations, we could assume that a particular solution $x$ of the nonhomogeneous problem~\eqref{nhp} is of the form $x=Xa$ where $a\in\cC(I,\bR^n)$ is a differentiable function. In that case, we have that, for problem~\eqref{nhp},
\[X'a+Xa'=x'=Ax+B\ol x+\gamma=AXa+B\ol X\ol a+\gamma.\]
In order to use the identity \eqref{pfm} to simplify this expression we have to assume that $a$ is even, that is, $\ol a=a$. If that is so, using expression~\eqref{pfm}, we deduce that
$Xa'=\gamma$ and therefore
\begin{equation}\label{aexp}a(t)=\int_{0}^tX^{-1}(s)\gamma(s)\dif s+c.\end{equation}
Where $c\in\bR^n$ is an arbitrary constant vector ($c=0$, for instance).
Therefore, the general solution to the nonhomogeneous equation~\eqref{nhp} would be
\[u(t)=X(t)a(t)+X(t)c=X(t)\left[\int_{0}^tX^{-1}(s)\gamma(s)\dif s+c\right],\]
where $c\in\bR^n$ is an arbitrary vector.\par
There is a clear inconsistency in this chain of thought: the assumption of $a$ being an even function is gratuitous for, in general, $X^{-1}\gamma$ needs not to be an odd function --something necessary for $a$ to be even according to formula ~\eqref{aexp}. This problem motivates a generalization of the method of variation of parameters in a way that allows us to tackle this problem.
\subsection{Second attempt: general method}

We consider now two problems:
\begin{equation}\label{hlsystemnh}Fu'(t)+Gu'(-t)+A u(t)+Bu(-t)=\c, t\in\bR,
\end{equation}
and an associated problem
\begin{equation}\label{hlsystemnh2}Fu'(t)-Gu'(-t)+A u(t)-Bu(-t)=\c, t\in\bR.
\end{equation}
Let us assume that $X$ and $Y$ are, respectively, fundamental matrices of the problems \eqref{hlsystemnh} and \eqref{hlsystemnh2}. We will consider solutions of the form $x=Xa+Yb$ where $a,b\in\cC(I,\bR^n)$ are differentiable functions such that $a$ is even and $b$ is odd. This assumptions are  similar to the ones exploited in \cite[Theorem 2.1]{Toj3} to obtain the Green's function. Then
\begin{align*}\gamma=  & Fx'+G\ol{x'}+Ax+B\ol x= FX'a+FXa'+FY'b+FYb'+G\ol{X'}a-G\ol Xa'\\&-G\ol{Y'}b+G\ol Yb'+AXa+B\ol X a+AYb-B\ol Yb= FXa'+FYb'-G\ol Xa'+G\ol Yb'.\end{align*}
That is, 
\begin{equation*}(FX-G\ol X)a'+(FY+G\ol Y)b'=\gamma.\end{equation*}
Considering the even and odd parts of the equation, we arrive to the system of equations
\begin{equation*}\begin{aligned}(F-G)X_ea'+(F-G)Y_ob' & =\gamma_o, \\ (F+G)X_oa'+(F+G)Y_eb' & =\gamma_e,
	\end{aligned}\end{equation*}
that is, assuming $F+G$ and $F-G$ are invertible,
\begin{equation}\label{systvpf}\begin{aligned}X_ea'+Y_ob' & =(F-G)^{-1}\gamma_o, \\ X_oa'+Y_eb' & =(F+G)^{-1}\gamma_e,
\end{aligned}\end{equation}

 In order to solve system~\eqref{systvpf}, we have to ask for the  associated matrix,
\begin{equation}\label{nXd}\bm\cX=\(\begin{array}{c|c}
   X_e & Y_o\\ \hline
  X_o & Y_e
   \end{array}\)\end{equation}
  to be invertible. \par
  
 The following results will help us to obtain sufficient criteria for $\bm\cX$ to be invertible. 
\begin{rem}\label{remdif}Define $M_+:=(F+G)^{-1}(A+B)$, $M_-:=(F-G)^{-1}(A-B)$. In the case we choose $X$ and $Y$ as given by Theorem \ref{thmexpfm},  we arrive to
\begin{align*}X(t)  = & \sum_{k=0}^\infty\frac{E^k t^{2k}}{(2k)!}  -M_+\sum_{k=0}^\infty\frac{E^k t^{2k+1}}{(2k+1)!},\ E=M_-M_+,\\ Y(t)= & \sum_{k=0}^\infty\frac{\widetilde E^k t^{2k}}{(2k)!}  -M_-\sum_{k=0}^\infty\frac{\widetilde E^k t^{2k+1}}{(2k+1)!},\ \widetilde E=M_+M_- .\end{align*}

Hence,
\renewcommand\arraystretch{1.5}
\[\bm\cX(t)=\sum_{k=0}^\infty\(\begin{array}{c|c}
   \frac{E^k t^{2k}}{(2k)!} & -M_-\frac{\widetilde E^k t^{2k+1}}{(2k+1)!}\\ \hline
  -M_+\frac{E^k t^{2k+1}}{(2k+1)!} & \frac{\widetilde E^k t^{2k}}{(2k)!}
   \end{array}\)
   =\sum_{k=0}^\infty\(\begin{array}{c|c}
      \frac{E^k t^{2k}}{(2k)!} & -\frac{ E^k t^{2k+1}}{(2k+1)!}M_-\\ \hline
     -\frac{\widetilde E^k t^{2k+1}}{(2k+1)!}M_+ & \frac{\widetilde E^k t^{2k}}{(2k)!}
      \end{array}\),\]
      and $\bm\cX(0)=\Id$, so $\bm\cX$ is invertible in a neighborhood of zero.\par 
  On the other hand, 
\begin{align*}\(\begin{array}{c|c}
  0 & M_-\\ \hline
  M_+ & 0
   \end{array}\)\bm\cX(t)  = & \sum_{k=0}^\infty\(\begin{array}{c|c} 
        -E\frac{E^k t^{2k+1}}{(2k+1)!} & M_-\frac{\widetilde E^k t^{2k}}{(2k)!}   \\ \hline
      M_+\frac{E^k t^{2k}}{(2k)!} & -\widetilde E\frac{\widetilde E^k t^{2k+1}}{(2k+1)!}
      \end{array}\)=\sum_{k=0}^\infty\(\begin{array}{c|c}
        -\frac{E^k t^{2k+1}}{(2k+1)!}M_-M_+ & \frac{ E^k t^{2k}}{(2k)!}M_-\\  \hline
   \frac{\widetilde E^k t^{2k}}{(2k)!}M_+ & -\frac{\widetilde E^k t^{2k+1}}{(2k+1)!}M_+M_-\\ 
   \end{array}\) \\ = & \bm\cX(t)\(\begin{array}{c|c}
     0 & M_-\\ \hline
     M_+ & 0
      \end{array}\).\end{align*}
   \renewcommand\arraystretch{1.2}
Since both matrices commute, there exists a simultaneous triangularization.
\begin{thm}[{\cite[Theorem 2.3.3]{Horn}}] Let $\cF\subset\cM_n(\bC)$ be a nonempty commuting family. There is a unitary matrix $U\in\cM_n(\bC)$ such that $U^*AU$ is upper triangular for every $A\in\cF$.
\end{thm}

\end{rem}
 \begin{lem}\label{constdet} If $F-G$, $F+G$, $A-B$ and $A+B$ are invertible and $(F-G)^{-1}(A-B)$, $(F+G)^{-1}(A+B)$ and $X(0)$ commute, then $X_e$ and $X_o$ commute and $\left|\bm\cX \right|=|X_eY_e-X_oY_o|$.
 \end{lem}
 \begin{proof} We are under the hypotheses of Corollary \ref{corfmu}, so every fundamental matrix of problem~\eqref{hlsystem} is of the form
\[X(t)=\left[\cosh \Omega t -(F+G)^{-1}(A+B)\Omega^{-1}\sinh\Omega t\right]X(0).\]
Hence,
\begin{align*}
X_e(t) & =\cosh \left(\Omega t\right)X(0),\\
X_o(t) & =-(F+G)^{-1}(A+B)\Omega^{-1}\sinh(\Omega t)X(0),
\end{align*}
and
\begin{align*}
X_e(t)X_o(t) & =-\cosh (\Omega t)X(0)(F+G)^{-1}(A+B)\Omega^{-1}\sinh(\Omega t)X(0),\\
X_o(t)X_e(t) & =-(F+G)^{-1}(A+B)\Omega^{-1}\sinh(\Omega t)X(0)\cosh( \Omega t)X(0),
\end{align*}
By construction, $\Omega$ (and the functions evaluated on $\Omega$) commutes with everything $(F-G)^{-1}(A-B)(F+G)^{-1}(A+B)$ commutes with. In particular,
\begin{align*}
(F-G)^{-1}(A-B)(F+G)^{-1}(A+B)X(0) & =X(0)(F-G)^{-1}(A-B)(F+G)^{-1}(A+B),\\
(F-G)^{-1}(A-B)(F+G)^{-1}(A+B)^2 & =(F+G)^{-1}(A+B)(F-G)^{-1}(A-B)(F+G)^{-1}(A+B),
\end{align*}
Thus, it is clear that $X_e$ and $X_o$ commute.\par
By Lemma \ref{lemdetmat}.1, since $X_e$ and $X_o$ commute, $\left|\bm\cX \right|=|X_eY_e-X_oY_o|$.
 \end{proof}
 
\begin{thm}\label{thmconstdet}Assume $A+B$, $A-B$, $F+ G$ and $F-G$ are invertible, and that $X(0)$, $(F-G)^{-1}(A-B)$ and  $(F+G)^{-1}(A+B)$ commute. Then, $X_eY_e-X_oY_o=X(0)Y(0)$ and $|\bm\cX|=|X(0)Y(0)|\ne 0$.
\end{thm}
\begin{proof}
 We know from Corollary \ref{corfmu} and Theorem \ref{thmexpfm} that
 \begin{align*}X(t) & =\(\cosh \Omega t -(F+G)^{-1}(A+B)\Omega^{-1}\sinh\Omega t\)X(0), \\
 Y(t) & =\(\cosh \Omega t -(F-G)^{-1}(A-B)\Omega^{-1}\sinh\Omega t\)Y(0).
 \end{align*}
Furthermore,
 \begin{align*}X_e(t) & =\cosh (\Omega t)X(0),\\
  X_o(t) & = -(F+G)^{-1}(A+B)\Omega^{-1}\sinh(\Omega t)X(0),\\
 Y_e(t) & =\cosh (\Omega t)Y(0),\\
  Y_o(t) & =-(F-G)^{-1}(A-B)\Omega^{-1}\sinh(\Omega t)Y(0).\end{align*}
Hence,
\begin{align*} & X_eY_e-X_oY_o \\ = & \cosh (\Omega t)X(0)\cosh (\Omega t)Y(0)-(F+G)^{-1}(A+B)\Omega^{-1}\sinh(\Omega t)X(0)(F-G)^{-1}(A-B)\Omega^{-1}\sinh(\Omega t)Y(0)
\\ = & \cosh^2 (\Omega t)X(0)Y(0)-(F+G)^{-1}(A+B)\Omega^{-1}\sinh(\Omega t)(F-G)^{-1}(A-B)\Omega^{-1}\sinh(\Omega t)X(0)Y(0)
\\ = & \left[\cosh^2 (\Omega t)-(F+G)^{-1}(A+B)\Omega^{-1}(F-G)^{-1}(A-B)\Omega^{-1}\sinh^2(\Omega t)\right]X(0)Y(0).
 \end{align*}
 Now, using that \[\Omega^2=(F-G)^{-1}(A-B)(F+G)^{-1}(A+B)=(F+G)^{-1}(A+B)(F-G)^{-1}(A-B),\] --because
 $(F+G)^{-1}(A+B)$ and  $(F-G)^{-1}(A-B)$ commute--, we have that
 \begin{align*}\Omega^{-1}(F-G)^{-1}(A-B)\Omega^{-1}& =\Omega^{-1}(F-G)^{-1}(A-B)\left[(F+G)^{-1}(A+B)(F-G)^{-1}(A-B)\right]^{-1}\Omega \\
 &=\Omega^{-1}(A+B)^{-1}(F+G)\Omega=\Omega^{-1}\Omega(A+B)^{-1}(F+G)=(A+B)^{-1}(F+G),  \end{align*}
 Therefore,
 \[(F+G)^{-1}(A+B)\Omega^{-1}(F-G)^{-1}(A-B)\Omega^{-1}=(F+G)^{-1}(A+B)(A+B)^{-1}(F+G)=\Id.\]
 Thus,\[X_eY_e-X_oY_o=\(\cosh^2\Omega t-\sinh^2\Omega t\)X(0)Y(0)=X(0)Y(0).\]
\end{proof}
\begin{rem} The commutativity of $(F-G)^{-1}(A-B)$ and  $(F+G)^{-1}(A+B)$ is necessary in the hypotheses of Theorem \ref{thmconstdet}. To see this, consider problem \eqref{hlsystemnh} with
\[F=\Id,\ G=0,\ A=
\left(
\begin{array}{cc}
 \frac{1}{2} & 1 \\
 1 & 0 \\
\end{array}
\right)
,\ B=
\left(
\begin{array}{cc}
 -\frac{1}{2} & 0 \\
 0 & 0 \\
\end{array}
\right)
.\]
Then,
\[(F+G)^{-1}=(F-G)^{-1}=\Id,\ A+B=\(\begin{array}{cc}
 0 & 1 \\
 1 & 0 \\
\end{array}
\right),\ A-B=\(\begin{array}{cc}
 1 & 1 \\
 1 & 0 \\
\end{array}
\right),\]
and
\[(A+B)(A-B)=\left(
\begin{array}{cc}
 1 & 0 \\
 1 & 1 \\
\end{array}
\right)\ne \left(
\begin{array}{cc}
 1 & 1 \\
 0 & 1 \\
\end{array}
\right)=(A-B)(A+B).\]
Also,
\[\Omega=\sqrt{(A-B)(A+B)}=\left(
\begin{array}{cc}
 1 & \frac{1}{2} \\
 0 & 1 \\
\end{array}
\right),\ \widetilde\Omega=\sqrt{(A+B)(A-B)}=\left(
\begin{array}{cc}
 1 & 0 \\
 \frac{1}{2} & 1 \\
\end{array}
\right).\]
So we have that
\[\bm\cX(t)=\left(
\begin{array}{cccc}
 \cosh (t) & \frac{1}{2} t \sinh (t) & \sinh (t) & \frac{1}{2} t \cosh (t)-\frac{\sinh (t)}{4} \\
 0 & \cosh (t) & 0 & \sinh (t) \\
 \sinh (t) & 0 & \cosh (t) & 0 \\
 \frac{1}{2} t \cosh (t)-\frac{\sinh (t)}{4} & \sinh (t) & \frac{1}{2} t \sinh (t) & \cosh (t) \\
\end{array}
\right).\]
Now,
\[|\bm\cX(t)|=\frac{1}{128} \left(-32 t^2+16 t \sinh (2 t)-\cosh (4 t)+129\right).\]
Observe that $|\bm\cX(0)|=1$ but $|\bm\cX(2)|= -4.81408\dots$ Also, in the neighbourhood  $[-1,1]$ the determinant $|\bm\cX(t)|$ is \emph{almost} constant equal to one which suggest and expression similar to a partial series approximation of the constant $1$.
\end{rem}
We summarize the method of variation of parameters in the following theorem.
\begin{thm}[Variation of Parameters Formula] \label{vpft}Assume $F+G$ and $F-G$ are invertible. Let $X$ and $Y$ be fundamental matrices of problems \eqref{hlsystemnh} and \eqref{hlsystemnh2} respectively and $\bm\cX$ defined as in \eqref{nXd}. Then the solutions of problem \eqref{hlsystemnh}, in a neighborhood of zero, are of the form
\begin{equation}\label{vpf}
u(t)=X(t)c+\(\begin{array}{c|c} X(t) & Y(t)\end{array}\)\int_0^t{\bm\cX(s)}^{-1}\begin{pmatrix}(F-G)^{-1}\c_o(s) \\ \hline  (F+G)^{-1}\c_e(s)\end{pmatrix}\dif s,
\end{equation}
where $c\in\bR^n$.
\end{thm}
\begin{proof}
Expression \eqref{vpf} is obtained from direct integration of \eqref{systvpf}.  $X(0)$ and $Y(0)$ are invertible by Corollary \ref{fmu} and Theorem \ref{fme}. Furthermore, it is clear that $X_o(0)=Y_o(0)=0$ and, therefore, $X_e(0)=X(0)$, $Y_e(0)=Y(0)$. Hence $\bm\cX(t)$ is invertible at zero and, by continuity, in a neighborhood of zero, so we conclude that expression \eqref{vpf} is well posed. That is, for every $c\in\bR^n$, \eqref{vpf} provides a solution of equation \eqref{hlsystemnh}.\par
On the other hand, if $v$ is a solution of problem \eqref{hlsystemnh}, consider $u$ to be the solution obtained from \eqref{vpf} for $c=X(0)^{-1}v(0)$. We have that $u(0)=v(0)$, which implies, by Lemma \ref{lemuos}, that $u=v$.
\end{proof}

\section{The Green's function}\label{sec6}

\subsection{The Green's function of the initial value problem}
Consider now problem \eqref{hlsystemnh} coupled with initial conditions
\begin{align}\label{pgf1}Fu'(t)+Gu'(-t)+A u(t)+Bu(-t) & =\c,\ t\in \bR,\\ \label{pgf2} u(0) & =\d,
\end{align}
where $A,B,F,G\in\cM_n(\bR)$,  $\c\in\cC(\bR)$, and $\d\in\bR^n$. 
If $F+G$ and $F-G$, are invertible, we know there exists a fundamental matrix $X$ of equation \eqref{pgf1}, so it is enough to take $X$ as in Theorem \ref{thmexpfm} ($X(0)=\Id$) and substitute $c$ by $\d$ in \eqref{vpf}. The next Theorem establishes this solution in terms of the Green's function.
\begin{thm}[Green's function] Assume $F+G$ and $F-G$ are invertible, $X$ and $Y$ are fundamental matrices of problems  \eqref{hlsystemnh} and \eqref{hlsystemnh2} respectively and $\bm\cX$ is invertible in $\bR$. Then problem \eqref{pgf1}--\eqref{pgf2} has a unique solution $u:\bR\to\bR^n$ and it is given by \[u(t)=X(t)X(0)^{-1}\d+\int_{-t}^tG(t,s)\c(s)\dif s,\] where
\[G(t,s)=\begin{dcases} \frac{1}{2}\(\begin{array}{c|c} X(t) & Y(t)\end{array}\){\bm\cX(s)}^{-1}\begin{pmatrix} (F-G)^{-1} \\ \hline (F+G)^{-1}\end{pmatrix}, & 0\le s\le t,\\
\frac{1}{2}\(\begin{array}{c|c} X(t) & Y(t)\end{array}\){\bm\cX(-s)}^{-1}\begin{pmatrix} -(F-G)^{-1} \\ \hline (F+G)^{-1}\end{pmatrix}, & -t\le s<0.
\end{dcases}\]
\end{thm}

\begin{proof} Following Theorem \ref{vpft}, we have that
\[u(t)=X(t)X(0)^{-1}\d+\(\begin{array}{c|c} X(t) & Y(t)\end{array}\)\int_{0}^t{\bm\cX(s)}^{-1}\begin{pmatrix}(F-G)^{-1}\c_o(s) \\ \hline (F+G)^{-1}\c_e(s)\end{pmatrix}\dif s.
\]
Hence, 
\begin{align*}
& u(t)-X(t)X(0)^{-1}\d \\ = &  \(\begin{array}{c|c} X(t) & Y(t)\end{array}\)\int_{0}^t{\bm\cX(s)}^{-1}\frac{1}{2}\left[\begin{pmatrix} (F-G)^{-1}\c(s) \\ \hline (F+G)^{-1}\c(s)\end{pmatrix}+\begin{pmatrix} -(F-G)^{-1}\ol\c(s) \\ \hline (F+G)^{-1}\ol\c(s)\end{pmatrix}\right]\dif s \\ = & 
\frac{1}{2}\(\begin{array}{c|c} X(t) & Y(t)\end{array}\)\left[\int_{0}^t{\bm\cX(s)}^{-1}\begin{pmatrix} (F-G)^{-1}\c(s) \\ \hline (F+G)^{-1}\c(s)\end{pmatrix}\dif s  +\int_{-t}^0{\bm\cX(-s)}^{-1}\begin{pmatrix} -(F-G)^{-1}\c(s) \\ \hline (F+G)^{-1}\c(s)\end{pmatrix}\dif s\right]
\end{align*}
Thus,
\[G(t,s)=\begin{dcases} \frac{1}{2}\(\begin{array}{c|c} X(t) & Y(t)\end{array}\){\bm\cX(s)}^{-1}\begin{pmatrix} (F-G)^{-1} \\ \hline (F+G)^{-1}\end{pmatrix}, & 0\le s\le t,\\
\frac{1}{2}\(\begin{array}{c|c} X(t) & Y(t)\end{array}\){\bm\cX(-s)}^{-1}\begin{pmatrix} -(F-G)^{-1} \\ \hline (F+G)^{-1}\end{pmatrix}, & -t\le s<0.
\end{dcases}\]
\end{proof}
\begin{exa}\label{exagf} We retake the simplified version of the problem studied in Example \ref{exasen}. That is, we consider the system of equations
\begin{equation}\label{ejfgi}\begin{aligned} u'(t) & = \b^2\, v(-t), \\ v'(t), & =\c^2\, u(-t),\end{aligned}
\end{equation}
with $\b,\c>0$. We have that
\begin{equation*}X(t)=
\begin{pmatrix}
 \cos \left( \beta\gamma t\right) & \frac{\b}{\c} \sin \left(\beta \gamma  t\right) \\
 \frac{\c}{\b} \sin \left(\beta \gamma t\right)  & \cos \left( \beta \gamma t\right)
\end{pmatrix}.\end{equation*}
Observe that the associated problem \eqref{hlsystemnh2}, in this case, is the same as \eqref{ejfgi} but substituting $\b^2$ and $\c^2$ by $-\b^2$ and $-\c^2$ respectively so, being careful with the signs,  we can conclude that
\begin{equation*}Y(t)=
\begin{pmatrix}
 \cos \left( \beta\gamma t\right) & -\frac{\b}{\c} \sin \left(\beta \gamma  t\right) \\
 -\frac{\c}{\b} \sin \left(\beta \gamma t\right)  & \cos \left( \beta \gamma t\right)
\end{pmatrix},\end{equation*}
 that is, $Y(t)=X(-t)$. Therefore
\[\bm\cX(t)= \begin{pmatrix}
 \cos \left( {\beta\gamma }t\right) &0 & 0& -\frac{\b}{\c} \sin \left({\beta \gamma } t\right) \\
 0& \cos \left( {\beta \gamma }t\right) & -\frac{\c}{\b} \sin \left({\beta \gamma } t\right)  & 0 \\0& \frac{\b}{\c} \sin \left({\beta \gamma } t\right) & \cos \left( {\beta\gamma }t\right) &0 \\
 \frac{\c}{\b} \sin \left({\beta \gamma } t\right)& 0 &  0& \cos \left( {\beta \gamma }t\right)
\end{pmatrix},\]
with inverse
\[\bm\cX(t)^{-1}= \begin{pmatrix}
 \cos \left( {\beta\gamma }t\right) &0 & 0& \frac{\b}{\c} \sin \left({\beta \gamma } t\right) \\
 0& \cos \left( {\beta \gamma }t\right) & \frac{\c}{\b} \sin \left({\beta \gamma } t\right)  & 0 \\0& -\frac{\b}{\c} \sin \left({\beta \gamma } t\right) & \cos \left( {\beta\gamma }t\right) &0 \\
 -\frac{\c}{\b} \sin \left({\beta \gamma } t\right)& 0 &  0& \cos \left( {\beta \gamma }t\right)
\end{pmatrix}.\]
Observe again that $\bm\cX(-t)=\bm\cX(t)^{-1}$. In this case,
\[G(t,s)=\begin{dcases} \left(
\begin{array}{cc}
 \cos ((s-t) \beta  \gamma ) & 0 \\
 0 & \cos ((s-t) \beta  \gamma ) \\
\end{array}
\right), & 0\le s\le t,\\
\left(\begin{array}{cc}
 0 & -\frac{\beta  \sin ((s+t) \beta  \gamma )}{\gamma } \\
 -\frac{\gamma  \sin ((s+t) \beta  \gamma )}{\beta } & 0 \\
\end{array}
\right), & -t\le s<0.
\end{dcases}\]
\end{exa}
\subsection{The Green's function of the boundary value problem}
Consider now problem \eqref{hlsystemnh} in the interval $I:=[-T,T]$ for some $T\in\bR^+$ coupled with two-point boundary conditions:
\begin{align}\label{pgf3}Fu'(t)+Gu'(-t)+A u(t)+Bu(-t) & =\c,\ t\in I,\\ \label{pgf4} Cu(-T)+Ku(T) & =\d,
\end{align}
where $A,B,C,F,G,K\in\cM_n(\bR)$,  $\c\in\cC(I)$, and $\d\in\bR^n$. We proceed in a similar way as in \cite[Chapter 1]{CabLibro}.
\begin{thm}\label{thmunbc}Assume $F+G$ and $F-G$ are invertible, $X$ and $Y$ are fundamental matrices of problems  \eqref{hlsystemnh} and \eqref{hlsystemnh2} respectively and $\bm\cX$, defined as in \eqref{nXd}, is invertible in $I$. Then problem \eqref{pgf3}--\eqref{pgf4} has a unique solution $u$ if and only if  $M_X=CX(-T)+KX(T)$ is invertible.
\end{thm}
\begin{proof}
By Theorem \ref{fmu}, there exists a fundamental matrix $X$ of problem \eqref{pgf3}. By Theorem \ref{vpft}, there exists a unique solution $u$ defined on $I$ given by \eqref{vpf} for some $c\in\bR^n$. Evaluating in $-T$ and $T$,
\begin{align*}
u(-T) & =X(-T)c+\(\begin{array}{c|c} X(-T) & Y(-T)\end{array}\)\int_0^{-T}{\bm\cX(s)}^{-1}\begin{pmatrix}(F-G)^{-1}\c_o(s) \\ \hline  (F+G)^{-1}\c_e(s)\end{pmatrix}\dif s, \\
u(T) & =X(T)c+\(\begin{array}{c|c} X(T) & Y(T)\end{array}\)\int_0^{T}{\bm\cX(s)}^{-1}\begin{pmatrix}(F-G)^{-1}\c_o(s) \\ \hline  (F+G)^{-1}\c_e(s)\end{pmatrix}\dif s, \\
\end{align*}

 If we impose the boundary conditions on $u$, we have that
 \begin{align*}
\d = & Cu(-T)+Ku(T) \\ = & C\left[X(-T)c+\(\begin{array}{c|c} X(-T) & Y(-T)\end{array}\)\int_0^{-T}{\bm\cX(s)}^{-1}\begin{pmatrix}(F-G)^{-1}\c_o(s) \\ \hline  (F+G)^{-1}\c_e(s)\end{pmatrix}\dif s\right]\\ & +K\left[X(T)c+\(\begin{array}{c|c} X(T) & Y(T)\end{array}\)\int_0^{T}{\bm\cX(s)}^{-1}\begin{pmatrix}(F-G)^{-1}\c_o(s) \\ \hline  (F+G)^{-1}\c_e(s)\end{pmatrix}\dif s\right] \\ = & M_Xc+ C\(\begin{array}{c|c} X(-T) & Y(-T)\end{array}\)\int_0^{-T}{\bm\cX(s)}^{-1}\begin{pmatrix}(F-G)^{-1}\c_o(s) \\ \hline  (F+G)^{-1}\c_e(s)\end{pmatrix}\dif s\\ & +K\(\begin{array}{c|c} X(T) & Y(T)\end{array}\)\int_0^{T}{\bm\cX(s)}^{-1}\begin{pmatrix}(F-G)^{-1}\c_o(s) \\ \hline  (F+G)^{-1}\c_e(s)\end{pmatrix}\dif s
\end{align*}
This linear equation has a unique solution for $c$ if and only if $M_X$ is invertible.
\end{proof}
\begin{thm}[Green's function] \label{gftpp}Assume $F-G$ and $F+G$ are invertible, $X$ and $Y$ are fundamental matrices of problems \eqref{hlsystemnh} and \eqref{hlsystemnh2} respectively, $\bm\cX$ is invertible in $I$ and $M_X$ is invertible. Then problem \eqref{pgf3}--\eqref{pgf4} has a unique solution $u$ and it is given by \begin{equation} \label{sgfbvp} u(t)=X(t)M_X^{-1}\d+\int_{-T}^TG(t,s)\c(s)\dif s,\end{equation} where
$G(t,s)=$
\begin{equation}\label{gfbvp}
\frac{1}{2}\begin{cases}
\begin{aligned} & \left[-X(t)M_X^{-1} K\(\begin{array}{c|c} X(T) & Y(T)\end{array}\) +\(\begin{array}{c|c} X(t) & Y(t)\end{array}\)\right]  {\bm\cX(s)}^{-1}\begin{pmatrix} (F-G)^{-1}\c(s) \\ \hline (F+G)^{-1}\c(s)\end{pmatrix}\\ & +  X(t)M_X^{-1}C\(\begin{array}{c|c} X(-T) & Y(-T)\end{array}\) {\bm\cX(-s)}^{-1}\begin{pmatrix} -(F-G)^{-1}\c(s) \\ \hline (F+G)^{-1}\c(s)\end{pmatrix},\end{aligned} & 0<s<t,\\
\begin{aligned} & X(t)M_X^{-1}C\(\begin{array}{c|c} X(-T) & Y(-T)\end{array}\) {\bm\cX(s)}^{-1}\begin{pmatrix} (F-G)^{-1}\c(s) \\ \hline (F+G)^{-1}\c(s)\end{pmatrix}\\ &   \left[-X(t)M_X^{-1} K\(\begin{array}{c|c} X(T) & Y(T)\end{array}\) +\(\begin{array}{c|c} X(t) & Y(t)\end{array}\)\right] {\bm\cX(-s)}^{-1}\begin{pmatrix} -(F-G)^{-1}\c(s) \\ \hline (F+G)^{-1}\c(s)\end{pmatrix},\end{aligned} & 0<-s<t, \\
\begin{aligned} & -X(t)M_X^{-1} K\(\begin{array}{c|c} X(T) & Y(T)\end{array}\)   {\bm\cX(s)}^{-1}\begin{pmatrix} (F-G)^{-1}\c(s) \\ \hline (F+G)^{-1}\c(s)\end{pmatrix}\\ & +  \left[X(t)M_X^{-1}C\(\begin{array}{c|c} X(-T) & Y(-T)\end{array}\)-\(\begin{array}{c|c} X(t) & Y(t)\end{array}\)\right] {\bm\cX(-s)}^{-1}\begin{pmatrix} -(F-G)^{-1}\c(s) \\ \hline (F+G)^{-1}\c(s)\end{pmatrix},\end{aligned} & 0<s<-t,\\
\begin{aligned} & \left[X(t)M_X^{-1} C\(\begin{array}{c|c} X(-T) & Y(-T)\end{array}\) -\(\begin{array}{c|c} X(t) & Y(t)\end{array}\)\right]  {\bm\cX(s)}^{-1}\begin{pmatrix} (F-G)^{-1}\c(s) \\ \hline (F+G)^{-1}\c(s)\end{pmatrix}\\ & -X(t)M_X^{-1}K\(\begin{array}{c|c} X(T) & Y(T)\end{array}\) {\bm\cX(-s)}^{-1}\begin{pmatrix} -(F-G)^{-1}\c(s) \\ \hline (F+G)^{-1}\c(s)\end{pmatrix},\end{aligned} & 0<-s<-t,\\
\begin{aligned} & -X(t)M_X^{-1} K\(\begin{array}{c|c} X(T) & Y(T)\end{array}\)  {\bm\cX(s)}^{-1}\begin{pmatrix} (F-G)^{-1}\c(s) \\ \hline (F+G)^{-1}\c(s)\end{pmatrix}\\ & +X(t)M_X^{-1}C\(\begin{array}{c|c} X(-T) & Y(-T)\end{array}\) {\bm\cX(-s)}^{-1}\begin{pmatrix} -(F-G)^{-1}\c(s) \\ \hline (F+G)^{-1}\c(s)\end{pmatrix},\end{aligned} & |t|<s,\\
\begin{aligned} & X(t)M_X^{-1} C\(\begin{array}{c|c} X(-T) & Y(-T)\end{array}\)  {\bm\cX(s)}^{-1}\begin{pmatrix} (F-G)^{-1}\c(s) \\ \hline (F+G)^{-1}\c(s)\end{pmatrix}\\ & -X(t)M_X^{-1}K\(\begin{array}{c|c} X(T) & Y(T)\end{array}\) {\bm\cX(-s)}^{-1}\begin{pmatrix} -(F-G)^{-1}\c(s) \\ \hline (F+G)^{-1}\c(s)\end{pmatrix},\end{aligned} & |t|<-s.\\
\end{cases}
\end{equation}
\end{thm}
\begin{proof} Following Theorem \ref{vpft},
\[u(t)=X(t)c+\(\begin{array}{c|c} X(t) & Y(t)\end{array}\)\int_{0}^t{\bm\cX(s)}^{-1}\begin{pmatrix}(F-G)^{-1}\c_o(s) \\ \hline  (F+G)^{-1}\c_e(s)\end{pmatrix}\dif s,
\]
where $c$ is, by Theorem \ref{thmunbc},
\begin{align*}c= & M_X^{-1}\left[\d-C\(\begin{array}{c|c} X(-T) & Y(-T)\end{array}\)\int_0^{-T}{\bm\cX(s)}^{-1}\begin{pmatrix}(F-G)^{-1}\c_o(s) \\ \hline  (F+G)^{-1}\c_e(s)\end{pmatrix}\dif s\right.\\ &  \phantom{ M_X^{-1}[\ \d}-\left. K\(\begin{array}{c|c} X(T) & Y(T)\end{array}\)\int_0^{T}{\bm\cX(s)}^{-1}\begin{pmatrix}(F-G)^{-1}\c_o(s) \\ \hline  (F+G)^{-1}\c_e(s)\end{pmatrix}\dif s\right].\end{align*}
Hence, if $\chi_{[a,b]}$ is the indicator function of the interval $[a,b]$, and $\chi_a^b$ is the \emph{oriented} indicator function of limits $a$ and $b$, that is, if $b\ge a$, $\chi_a^b=\chi_{[a,b]}$, if $b<a$, $\chi_a^b=-\chi_{[b,a]}$, then
\begin{align*}
& u(t)-X(t)M_X^{-1}\d \\ = & -X(t)M_X^{-1}C\(\begin{array}{c|c} X(-T) & Y(-T)\end{array}\)\int_0^{-T}{\bm\cX(s)}^{-1}\begin{pmatrix}(F-G)^{-1}\c_o(s) \\ \hline  (F+G)^{-1}\c_e(s)\end{pmatrix}\dif s\\ & -X(t)M_X^{-1} K\(\begin{array}{c|c} X(T) & Y(T)\end{array}\)\int_0^{T}{\bm\cX(s)}^{-1}\begin{pmatrix}(F-G)^{-1}\c_o(s) \\ \hline  (F+G)^{-1}\c_e(s)\end{pmatrix}\dif s
\\ & +\(\begin{array}{c|c} X(t) & Y(t)\end{array}\)\int_{0}^t{\bm\cX(s)}^{-1}\begin{pmatrix}(F-G)^{-1}\c_o(s) \\ \hline  (F+G)^{-1}\c_e(s)\end{pmatrix}\dif s
\\ 
= & \int_{-T}^{T}\left(X(t)M_X^{-1}\left[C\(\begin{array}{c|c} X(-T) & Y(-T)\end{array}\)\chi_{[-T,0]}(s)- K\(\begin{array}{c|c} X(T) & Y(T)\end{array}\)\chi_{[0,T]}(s)\right] +\(\begin{array}{c|c} X(t) & Y(t)\end{array}\)\chi_{[0,t]}(s)\right) \\ & {\bm\cX(s)}^{-1}\begin{pmatrix}(F-G)^{-1}\c_o(s) \\ \hline  (F+G)^{-1}\c_e(s)\end{pmatrix}\dif s  .
\end{align*}
On the other hand,
\[\begin{pmatrix} (F-G)^{-1}\c_o \\ \hline  (F+G)^{-1}\c_e\end{pmatrix}=\frac{1}{2}\left[\begin{pmatrix} (F-G)^{-1}\c \\ \hline (F+G)^{-1}\c\end{pmatrix}+\begin{pmatrix} -(F-G)^{-1}\ol\c \\ \hline (F+G)^{-1}\ol\c\end{pmatrix}\right].\]
Therefore,
\begin{align*}
& u(t)-X(t)M_X^{-1}\d  \\ = & \frac{1}{2}\int_{-T}^{T}\left(X(t)M_X^{-1}\left[C\(\begin{array}{c|c} X(-T) & Y(-T)\end{array}\)\chi_{[-T,0]}(s)- K\(\begin{array}{c|c} X(T) & Y(T)\end{array}\)\chi_{[0,T]}(s)\right] +\(\begin{array}{c|c} X(t) & Y(t)\end{array}\)\chi_0^t(s)\right)\\ & {\bm\cX(s)}^{-1}\begin{pmatrix} (F-G)^{-1}\c(s) \\ \hline (F+G)^{-1}\c(s)\end{pmatrix}\dif s \\ & +  \frac{1}{2}\int_{-T}^{T}\left(X(t)M_X^{-1}\left[C\(\begin{array}{c|c} X(-T) & Y(-T)\end{array}\)\chi_{[-T,0]}(s)- K\(\begin{array}{c|c} X(T) & Y(T)\end{array}\)\chi_{[0,T]}(s)\right] +\(\begin{array}{c|c} X(t) & Y(t)\end{array}\)\chi_0^t(s)\right)\\ & {\bm\cX(s)}^{-1}\begin{pmatrix} -(F-G)^{-1}\ol\c(s) \\ \hline (F+G)^{-1}\ol\c(s)\end{pmatrix}\dif s \\
= & \frac{1}{2}\int_{-T}^{T}\left(X(t)M_X^{-1}\left[C\(\begin{array}{c|c} X(-T) & Y(-T)\end{array}\)\chi_{[-T,0]}(s)- K\(\begin{array}{c|c} X(T) & Y(T)\end{array}\)\chi_{[0,T]}(s)\right] +\(\begin{array}{c|c} X(t) & Y(t)\end{array}\)\chi_0^t(s)\right) \\ & {\bm\cX(s)}^{-1}\begin{pmatrix} (F-G)^{-1}\c(s) \\ \hline (F+G)^{-1}\c(s)\end{pmatrix}\dif s \\ & +  \frac{1}{2}\int_{-T}^{T}\left(X(t)M_X^{-1}\left[C\(\begin{array}{c|c} X(-T) & Y(-T)\end{array}\)\chi_{[-T,0]}(-s)- K\(\begin{array}{c|c} X(T) & Y(T)\end{array}\)\chi_{[0,T]}(-s)\right] +\(\begin{array}{c|c} X(t) & Y(t)\end{array}\)\chi_0^t(-s)\right)\\ & {\bm\cX(-s)}^{-1}\begin{pmatrix} -(F-G)^{-1}\c(s) \\ \hline (F+G)^{-1}\c(s)\end{pmatrix}\dif s
\\ =
& \frac{1}{2}\int_{-T}^{T}\left(X(t)M_X^{-1}\left[C\(\begin{array}{c|c} X(-T) & Y(-T)\end{array}\)\chi_{[-T,0]}(s)- K\(\begin{array}{c|c} X(T) & Y(T)\end{array}\)\chi_{[0,T]}(s)\right] +\(\begin{array}{c|c} X(t) & Y(t)\end{array}\)\chi_0^t(s)\right) \\ & {\bm\cX(s)}^{-1}\begin{pmatrix} (F-G)^{-1}\c(s) \\ \hline (F+G)^{-1}\c(s)\end{pmatrix}\\ & +  \left(X(t)M_X^{-1}\left[C\(\begin{array}{c|c} X(-T) & Y(-T)\end{array}\)\chi_{[0,T]}(s)- K\(\begin{array}{c|c} X(T) & Y(T)\end{array}\)\chi_{[-T,0]}(s)\right] +\(\begin{array}{c|c} X(t) & Y(t)\end{array}\)\chi_{-t}^0(s)\right)\\ & {\bm\cX(-s)}^{-1}\begin{pmatrix} -(F-G)^{-1}\c(s) \\ \hline (F+G)^{-1}\c(s)\end{pmatrix}\dif s.
\end{align*}
That is,  $G$ is as in \eqref{gfbvp}.
\end{proof}
\begin{rem}Observe that, in expression \eqref{gfbvp} we are defining $G$ in an open subset of $I^2$. The set in which $G(t,\cdot)$ has not been defined is of zero Lebesgue measure in $I$, so it is irrelevant in terms of the obtaining of the solution in equation  \eqref{sgfbvp}.
\end{rem}
\begin{exa} We retake the problem in Example \ref{exagf} the simplified version of the problem studied in Example \ref{exasen} adding the periodic boundary conditions $(u,v)(-T)=(u,v)(T)$ for some fixed $T\in\bR+$. We have that
\[M_X=X(-T)-X(T)=\left(
\begin{array}{cc}
 0 & -\frac{2 \beta  \sin (T \beta  \gamma )}{\gamma } \\
 -\frac{2 \gamma  \sin (T \beta  \gamma )}{\beta } & 0 \\
\end{array}
\right),\]
has to be invertible, that is, we have to take $T\in\bR^+\backslash\{(k\pi-\pi/2)/(\b\c)\ :\ k\in\bN\}$ in order to compute the Green's function. Also,
\[M_X^{-1}K\(\begin{array}{c|c} X(T) & Y(T)\end{array}\)=\left(
\begin{array}{cccc}
 \frac{1}{2} & \frac{\beta  \cot (T \beta  \gamma )}{2 \gamma } & -\frac{1}{2} & \frac{\beta  \cot (T \beta  \gamma )}{2 \gamma } \\
 \frac{\gamma  \cot (T \beta  \gamma )}{2 \beta } & \frac{1}{2} & \frac{\gamma  \cot (T \beta  \gamma )}{2 \beta } & -\frac{1}{2} \\
\end{array}
\right).\]

Therefore, 
\[2\sin (\beta  \gamma  T)\, G(t,s)=\begin{dcases}
 \begin{pmatrix}
  \sin ((s-t+T) \beta  \gamma ) & -\frac{\beta  \cos ((s+t-T) \beta  \gamma )}{\gamma } \\
  -\frac{\gamma  \cos ((s+t-T) \beta  \gamma )}{\beta } & \sin ((s-t+T) \beta  \gamma ) \\
\end{pmatrix}
 , & |s|< t,\\
\begin{pmatrix}
 \sin ((s-t-T) \beta  \gamma ) & -\frac{\beta  \cos ((s+t+T) \beta  \gamma )}{\gamma } \\
 -\frac{\gamma  \cos ((s+t+T) \beta  \gamma )}{\beta } & \sin ((s-t-T) \beta  \gamma ) \\
\end{pmatrix}
, & |s|<-t, \\
\begin{pmatrix}
 \sin ((s-t-T) \beta  \gamma ) & -\frac{\beta  \cos ((s+t-T) \beta  \gamma )}{\gamma } \\
 -\frac{\gamma  \cos ((s+t-T) \beta  \gamma )}{\beta } & \sin ((s-t-T) \beta  \gamma ) \end{pmatrix}, &  |t|<s,
\\
\begin{pmatrix}
 \sin ((s-t+T) \beta  \gamma ) & -\frac{\beta  \cos ((s+t+T) \beta  \gamma )}{\gamma } \\
 -\frac{\gamma  \cos ((s+t+T) \beta  \gamma )}{\beta } & \sin ((s-t+T) \beta  \gamma ) \\
\end{pmatrix}, & |t|< -s.
\end{dcases}\]
\end{exa}
\begin{exa}\label{le} The theory presented in this paper generalizes the theory of Green's functions of linear differential equations with reflection. Now we retake the classical problem studied in \cite{Cab4}.

\begin{equation}\label{cprob} x'(t)+mx(-t)=h(t),\ t\in I=[-T,T],\ x(-T)=x(T),\end{equation}
where $m>0$.
We have that, using formula \eqref{fme}, a fundamental matrix of problem \eqref{cprob} is
\[X(t)=\cos m t -\sin m t.\]
Correspondingly,
\[Y(t)=\cos m t +\sin m t.\]
Hence,
\[\bm\cX(t)=\begin{pmatrix} \cos mt & \sin mt \\ -\sin mt & \cos mt\end{pmatrix}.\]
Using the formula for the Green's function provided by Theorem  \eqref{gftpp},
\begin{equation*}\label{gbarra}
2\sin(mT) G(t,s)=\begin{cases} \cos m(T-s-t)+\sin m(T+s-t), & |s|<t,\\\cos m(T-s-t)-\sin m(T-s+t), & |t|<s,\\\cos m(T+s+t)+\sin m(T+s-t), &  |t|<-s,\\\cos m(T+s+t)-\sin m(T-s+t), & |s|<-t.\end{cases}
\end{equation*}
This expression coincides with the one obtained in \cite{Cab4} as expected.
\end{exa}

\section{Final remarks}

We would like to make some last comments on the preceding discussion. The reader may be missing the corresponding generalization of the previous results to the case of equations with nonconstant coefficients. Although in the case of ODEs this step is straightforward and the theory applies without any change of relevance, in the case of equations with reflection this generalization if far from trivial. In fact, it is not possible in general. We refer the reader to \cite{CabToj} for more information on this subject.
\par
On a different matter, the reader may have realized that most of what is done here is valid for linear differential equations with coefficients in Banach algebras with unity. Actually, matrices form a really poorly behaved Banach algebra. Most of the trouble we went trough in this paper originated in the facts that, on one hand, matrices do not commute in general --Theorems \ref{thmexpfm},  \ref{mpaf} and  \ref{thmconstdet}, Lemmas \ref{lempha} and \ref{constdet}, Remark \ref{remdif} highlight this issue-- and, on the other, that they conform an algebra with divisors of zero --which in this case correspond to the singular matrices. Therefore, the theory presented here may be extended to other Banach algebras of linear endomorphisms of given vector spaces.\par

\end{document}